\documentclass[11pt]{article}
\input{definitions}

\begin{document}

\numberwithin{equation}{section} \numberwithin{lemma}{section}


\begin{frontmatter}
\title{Modeling tagged pedestrian motion: a 
mean-field type game approach\thanks{Financial support from the 
Swedish 
Research 
Council (2016-04086) is 
gratefully acknowledged. \\
We thank the anonymous reviewers for comments and suggestions that greatly 
helped to improve the presentation of the results, and we thank E. Cristiani 
for directing us to \cite{cristiani2015modeling}.
} }
\author{Alexander Aurell \thanks{Department of Mathematics, KTH Royal Institute 
of Technology, 100 44 
Stockholm, Sweden. \newline E-mail address: aaurell@kth.se}
\qquad 
Boualem Djehiche \thanks{Department of Mathematics, KTH Royal Institute 
of Technology, 100 44 
Stockholm, Sweden. \newline E-mail address: boualem@kth.se}
}

\runtitle{Tagged pedestrian model}
\runauthor{Aurell, Djehiche}

\today

\skp

\begin{abstract}
This paper suggests a model for the motion of \textit{tagged pedestrians}: 
pedestrians moving towards a specified targeted destination, which they are 
forced to reach. It aims to be a decision-making tool for the 
positioning of fire fighters, security personnel and other services in a 
pedestrian environment. Taking interaction with the surrounding crowd into 
account leads to a differential nonzero-sum game model where the tagged 
pedestrians compete with the surrounding crowd of ordinary pedestrians. 
When deciding how to act, pedestrians consider crowd distribution-dependent 
effects, like congestion and crowd aversion. Including such effects in the 
parameters of the game, makes it a \textit{mean-field type game}. The 
equilibrium control is characterized, and special cases are discussed. Behavior 
in the model is studied by numerical simulations.

\noi {\bf MSC 2010: 49N90, 49J55, 60H10, 60K30, 91A80, 93E20}

\noi {\bf Keywords: pedestrian dynamics; backward-forward stochastic 
differential equations; mean-field type games; congestion; crowd 
aversion; evacuation planning}
\end{abstract}

\end{frontmatter}


\section{Introduction}
\label{sec:introduction}
Tagged pedestrians are individuals that plan their motion 
from an \textit{unspecified initial position} in order to reach a 
\textit{specified target location} in a certain time. The model for tagged 
pedestrian motion proposed in this paper is based on mean-field type game 
theory, and is a decision making tool for the positioning of fire fighters, 
medical personnel, etc, during mass gatherings.
The tagged's prime objective is the pre-set final destination 
which is considered essential to reach; ending up in proximity of the final 
destination is not acceptable, it has to be reached. This is in sharp 
contrast to the standard finite-horizon 
models cited below, where pedestrians are penalized if their final position 
deviates from a target position, such penalization is a 'soft' constraint and 
can be broken at a cost.
The tagged's initial position is chosen rationally. Therefore, we are 
inclined to think of the tagged as external entities to be deployed in the 
crowd. Where they (rationally) ought to be deployed is subject to an offline 
calculation made by a coordinator: a \textit{central planner}. Besides tagged 
pedestrian motion, possible applications of the model include cancer cell 
dynamics and smart medicine in the human body, and malware propagation in a 
network, among other.
\vspace{0.2cm}\\
The central planner's decision making is based on knowledge of 
the pedestrian distribution. As noted in \cite{naldi2010mathematical}, the 
pedestrian behavior in dense crowds is empirically random to some extent, 
likely due to the large number of external inputs. In a noisy environment, the 
central planner anticipates the behavior of the crowd and predicts the tagged's 
path to the target. As is standard in the mean-field approach, interaction with 
tagged and ordinary pedestrians is modeled as reactions to the 
state distribution of a representative tagged and ordinary pedestrian, 
respectively. This leads us to formulate a mean-field type game based
model, which in certain scenarios reduces to an optimal control based model.

\subsection{Related work}

\subsubsection{Optimal control and games of mean-field type}

Rational pedestrian behavior is, in this paper, characterized by either a game 
equilibrium, or an optimal control. 
The tool used to find the equilibrium/optimal behavior is Pontryagin's 
stochastic maximum 
principle (SMP). For stochastic control problems, SMP yields, when available, 
necessary conditions that must be satisfied by any 
optimal solution. The necessary conditions become sufficient under additional 
convexity conditions. Early results show that an optimal control along 
with the corresponding optimal state trajectory must solve the so-called 
Hamiltonian system, which is a two-point (forward-backward) boundary value 
problem, together with a maximum condition of the so-called Hamiltonian 
function (see~\cite{yong1999stochastic} for a 
detailed account). 
In stochastic differential games, both zero-sum and nonzero-sum, SMP is one of 
the main tools for obtaining conditions for an equilibrium, essentially 
inherited from the theory of stochastic optimal control. For recent examples 
of the use of SMP in stochastic differential game theory, see 
\cite{moon2018risk, aurell2018mean}.
\vspace{0.2cm}\\
In stochastic systems, the backward equation is 
fundamentally different from the forward equation, if the solution is required 
to be adapted. An adapted solution to a backward stochastic differential 
equation (BSDE) is a pair of adapted stochastic processes $(Y_\cdot, Z_\cdot)$, 
where $Z_\cdot$ corrects any ``non-adaptedness'' caused by the terminal 
condition of $Y_\cdot$. As pointed out in \cite{kohlmann2000relationship}, the 
first component~$Y_\cdot$ corresponds to the mean evolution of the dynamics, 
and $Z_\cdot$ to the risk between current time and terminal time. Linear BSDEs 
extend to non-linear BSDEs \cite{pardoux1990adapted} and backward-forward SDEs 
(BFSDE) \cite{antonelli1993backward,hu1995solution}. 
BSDEs with distribution-dependent coefficients, mean-field BSDEs, are by now 
well-understood objects \cite{buckdahn2009mean, buckdahn2009meanb}. 
Mean-field BFSDEs arise naturally in the probabilistic analysis of mean-field 
games (MFG), mean-field type games (MFTG) and optimal control of mean-field 
type equations.
\vspace{0.2cm}\\
The theory of optimal control of mean-field SDEs, initiated in 
\cite{andersson2011maximum}, treats stochastic control problems with 
coefficients dependent on the marginal state-distribution. This theory is by 
now well developed for forward stochastic dynamics, i.e., with initial 
conditions on state \cite{bensoussan2013mean, djehiche2015stochastic, 
buckdahn2016stochastic, carmona2018probabilistic}. With SMPs for optimal 
control problems of mean-field type 
at hand, MFTG theory can inherit these techniques, just like 
stochastic differential game theory does in the mean-field free case. See 
\cite{tembine2017mean} for a review of solution approaches to MFTGs. 
\vspace{0.2cm}\\
Optimal control of mean-field BSDEs has recently gained attention. In 
\cite{li2016linear} the mean-field LQ BSDE control problem with 
deterministic coefficients is studied. Assuming 
the control space is linear, linear perturbation is used to derive a 
stationarity condition which together with a mean-field FBSDE system 
characterizes the optimal control. Other recent work on the control of 
mean-field BSDEs makes use of the SMP approach 
of \cite{yong2010forward} to control of BFSDEs.
\vspace{0.2cm}\\
Optimal control problems of mean-field type can be interpreted as large 
population limits of cooperative games, 
where the players collaborate to optimize a joint objective 
\cite{lacker2017limit}. A close relative to mean-field type control is 
MFG, a class of non-cooperative stochastic differential 
games, initiated by \cite{huang2006large} and \cite{lasry2007mean} 
independently. For both mean-field type control problems and MFG, the 
games approximated are games between a large number of indistinguishable 
(anonymous) players, interacting weakly through a mean-field coupling term.
Weak~player-to-player interaction through the mean-field coupling restricts the
influence one player has on any other player to be inversely proportional to 
the number of players, hence the level of influence of any specific player is 
very small. In contrast to the MFG, players in a MFTG can be influential, and 
distinguishable (non-anonymous). That is, state~dynamics and/or cost need not 
be of the same form over the whole player population, and~a~single player can 
have a major influence on other players' dynamics and/or cost.

\subsubsection{Pedestrian crowd modeling}
There is a variety of mathematical approaches to the 
modeling of pedestrian crowd motion. \emph{Microscopic 
force-based} models \cite{helbing1995social, chraibi2011force}, and in 
particular the social force model, represent pedestrian behavior as a reaction 
to forces and potentials, applied not only by the surrounding environment but 
also by the pedestrian's internal motivation and desire. A \emph{cellular 
automata} approach to microscopic modeling of pedestrian 
crowds can be found in \cite{burstedde2001simulation, 
jin2017simulating, ibrahim2017features}, to mention only few. 
\emph{Macroscopic} models view the crowd as a continuum, described by averaged 
quantities such as density and pressure. The Hughes model 
\cite{hughes2002continuum, huang2009revisiting, twarogowska2014macroscopic} 
couples a conservation law, representing the physics of the crowd, with an 
Eikonal equation modeling a common task of the pedestrians. Its variations are 
manifold. \emph{Kinetic} and other \emph{multi-scale} models 
\cite{bellomo2013microscale, cristiani2011multiscale, bellomo2016human} 
constitute an intermediate step between the micro- and the macro scales.  
\emph{Microscopic game} and \emph{optimal control} models for pedestrian 
crowd dynamics, with their relevant continuum limits in 
the form of \emph{mean-field games} and \emph{mean-field type optimal 
control}, have gained interest in the last decade. In 
\cite{hoogendoorn2004pedestrian}, a simplified MFG was used to model 
rational behavior of pedestrian in a crowd. Following Lasry and Lion's paper on 
MFG, \cite{dogbe2010modeling} applied MFG to pedestrian crowd 
motion and \cite{lachapelle2011mean} used MFG to model local congestion effects, 
i.e. the relationship between energy needed to walk/run and local crowd density.
MFG based models have also been used to simulate evacuation of pedestrians
\cite{burger2013mean, burger2014mean, djehiche2017mean}.
\vspace{0.2cm}\\
The mean-field approach rests on an exchangeability 
assumption; pedestrians are anonymous, they may have different paths but one 
individual cannot be distinguished from another. While interesting when 
modeling circumstances where pedestrians can be considered indistinguishable, 
for instance a train station during rush hour or fast exits of an area in case 
of an alarm, there are situations where an anonymous crowd model is not 
satisfactory. Mean-field type games is a tool to extend the mean-field 
approach to distinguishable sub-crowds and influential individuals, and was 
applied in \cite{aurell2017mean}. Other ways to break the anonymity 
within the mean-field approach are \emph{multi-population MFGs} 
\cite{feleqi2013derivation, cirant2015multi, achdou2017mean} and \emph{major 
agent models} \cite[and references therein]{huang2016backward}.
\vspace{0.2cm}\\
Another important characteristic of standard MFGs and MFTGs is the assumption of
anticipative players. Each player is assumed to predict how 
the whole population will act in the future, and then pick a strategy 
accordingly. In a pedestrian crowd setting, this would correspond to pedestrians 
knowing, most likely by experience, how the surrounding crowd will behave. 
This is a very high 'level of rationality', and certainly not appropriate in 
all scenarios. We refer to \cite{cristiani2015modeling} for a precise 
discussion on 
the level of rationality of pedestrian crowd models, including mean-field games 
and optimal control of mean-field type based models.

\subsection{Paper contribution and outline}

This paper investigates a new modeling approach to the motion of 
pedestrians whose primary objective is to reach a specific target, at a 
specific time. The model can represent a large group that is steered by a
central planner, or a single individual. Even though these pedestrians have 
certain non-standard goals, they are constrained by the same physical 
limitations as ordinary pedestrians, and the central planner takes this into 
consideration. Moreover, the central planner has access to complete information 
of the surrounding environment, and utilizes it in the decision making process.
\vspace{0.2cm}\\
The contribution of this paper is a mean-field type game based model for the 
motion of tagged pedestrians in a surrounding crowd of 'ordinary pedestrians'. 
The players in the game are \textit{crowds}, and act under 
general distribution-dependent dynamics and cost. The tagged pedestrians have a 
'hard' terminal condition, while the ordinary pedestrians have a 'hard' initial 
condition, and this results in a state equation in the form of a mean-field 
BFSDE, representing i) the predicted 
motion of the tagged towards the target, coupled with ii) the evolution of the 
surrounding crowd from a known initial configuration.
\vspace{0.2cm}\\
Rational behavior in the model is characterized with a version of SMP, tailored 
for the mean-field BFSDE system with mean-field type costs. The central section 
of this paper is the solved examples, where we illustrate pedestrian behavior 
in the model. Further directions of research are also outlined.
\vspace{0.2cm}\\
The tagged pedestrian model is presented in Section \ref{sec:BFSDE}, which 
begins in a deterministic setting, to which we gradually add components until 
the full model is reached. The SMP that gives necessary and sufficient 
conditions for a pair of equilibrium controls for the mean-field type game is 
presented in Theorem \ref{thm:game_smp}. Examples of tagged motion are studied 
in Section \ref{sec:distributed}. All technical proofs and background theory 
are moved to the appendix.

\section{The tagged pedestrian model}
\label{sec:BFSDE}
In this section the tagged pedestrian model is introduced. Velocity fields 
are the driving components in the model and include both small-scale pedestrian 
interactions and path planning components. The latter is implemented by 
pedestrians in a rational way: A cost functional summarizing pedestrian 
preferences is minimized. The former 
describes involuntary movement over which the pedestrian has no control. 
\subsection*{List of symbols}
\begin{itemize}
 \item[] $T\in (0.\infty)$ -- the time horizon.
 \item[] $(\Omega, \mathcal{F}, \mathbb{F}, \mathbb{P})$ -- the underlying 
  filtered probability space.
 \item[] $\mathcal{P}(\mathbb{R}^d)$ -- the space of probability measures on 
$(\mathbb{R}^d, \mathcal{B}(\mathbb{R}^d))$.
 \item[] $\mathcal{P}_2(\mathbb{R}^d)$ -- all square-integrable measures in 
$\mathcal{P}(\rd)$.
 \item[] $\law{X} \in \mathcal{P}(\mathbb{R}^k)$ -- the distribution of the 
random variable $X\in\mathbb{R}^k$ under $\mathbb{P}$.
 \item[] $X_\cdot$ -- the stochastic process $\{X_t;\, t\geq 0\}$.
 \item[] $L^2_{\mathcal{F}}(\rd)$ -- the space of $\mathcal{F}$-measurable 
square-integrable $\rd$-valued random variables.
\end{itemize}

\subsection{Pedestrian dynamics}

In a deterministic setting pedestrian state dynamics are 
described by ordinary differential equations (ODE). An ordinary pedestrian is 
initiated at some location $x_0\in\rd$, and moves according to an 
ODE with an initial condition,
\begin{equation}
\label{eq:det_dynamicsx}
\begin{cases}
 \frac{d}{dt}X_t = b^x_t, & t \in [0,T],
 \\
 X_0 = x_0.
\end{cases}
\end{equation}
The pedestrian influences its velocity through a \textit{control} function, 
$u^x_\cdot$. The control is assumed to take values in the set 
$U^x\subset\mathbb{R}^{d_x}$, $d_x\geq 1$. Alongside the control function, 
$b^x_t$ may depend on interaction with other pedestrians, for example through 
collisions. In the literature, the velocity is often split into a 
\textit{behavioral velocity} (the control) and an \textit{interaction velocity},
\begin{equation}
 b^x_t = u^x_t + b^{x,\text{int}}_t.
\end{equation}
So on top of any interaction 
velocity the tagged influences its movement through a control, and this grants 
it some smartness. As was discussed in the introduction, the pedestrian may 
foresee crowd movement and act in advance to avoid congested areas and other 
obstacles. Pedestrian models that considers the behavioral velocity to be an 
internal choice of the pedestrian leaves the framework of classical particle 
models and enters decision-based smart models. A summary of 
the difference between these model classes is found in 
\cite{naldi2010mathematical}.
Alongside ordinary pedestrians, \textit{tagged pedestrians} are deployed. 
They represent a person on a mission, who has to reach a target location $y_T 
\in 
\rd$ at time $t=T$. In the deterministic setting, the tagged moves according to 
an ODE with terminal 
condition,
\begin{equation}
\label{eq:det_dynamics}
 \begin{cases}
  \frac{d}{dt}Y_t = u^y_t + b^{y,\text{int}}_t,& t\in[0,T],\ u^y_t \in 
U^y\subset \mathbb{R}^{d_y},\ d_y\geq 1,
  \\
  Y_T = y_T.
 \end{cases}
\end{equation}
Just like $u_\cdot^x$ influences the terminal position of an ordinary 
pedestrian, so does $u_\cdot^y$ influence the initial position of a tagged 
pedestrian. This should be 
interpreted in the following way: the initial position of the tagged pedestrian 
is not pre-determined, but depends on the pedestrian's choice of behavioral 
velocity. This 
choice is subject to the terminal condition, and at the 
same time it adheres to other preferences. For example, if there is a high 
risk 
of injury at a certain location $y_T$ (doors, stairs, etc.), where is the best 
spot for a medic to be positioned, so that she can reach $y_T$ in time $T$?
Certainly not at $y_T$, since it is a high risk area. The medic's initial 
location is preferably a safe spot, from which it is easy for her to 
access $y_T$, taking surrounding pedestrians and environment into account. This 
is what should be reflected in the choice of $u_\cdot^y$. 
\vspace{0.2cm}\\
Pedestrian motion can be considered deterministic if the crowd is sparse, but 
partially random if the crowd is dense. 
To capture this, \eqref{eq:det_dynamicsx}-\eqref{eq:det_dynamics} is extended 
to its stochastic counterpart. Let
$(\Omega, \F, \p)$ be a complete probability space, endowed with the filtration 
$\f = (\F_t)_{t\geq 0}$, 
satisfying the usual conditions. Let the space carry $B^x_\cdot$ and 
$B^y_\cdot$, independent $w_x$- and $w_y$-dimensional $\f$-Wiener processes, 
and a random variable $y_T\in L^2_{\F_T}(\rd)$ independent of $B_\cdot := 
(B^x_\cdot, B^y_\cdot)^*$. The Brownian motion $B_\cdot$ is split into 
$B^x_\cdot$ and $B^y_\cdot$ to emphasize modeling features: $B^x_\cdot$ is the 
noise that explicitly effects the ordinary pedestrians diffusion, while 
$B^y_\cdot$ may be used to model any noise that in addition to 
$B^x_\cdot$ effects the tagged. All information in the model up to time $t$ is 
contained in $\F_t$, and a 
process that depends only on past and current information is called an 
$\mathcal{F}_t$-\textit{adapted process}. It is natural to require pedestrian 
motion to be adapted, since pedestrians 
react causally to the environment.  In this random environment, we consider a 
control to be feasible if it is open-loop adapted and square integrable, i.e. 
belongs to the sets $\mathcal{U}^x$ and $\mathcal{U}^y$ for ordinary and 
tagged pedestrians respectively,
\begin{equation}
\label{eq:def_feasible}
\begin{aligned}
 \mathcal{U}^x &:= \left\{ u : \Omega \times [0,T] \rightarrow U^x\ \big|\ 
u_\cdot \text{ is }\mathbb{F}\text{-adapted},\ \mathbb{E}\left[\int_0^T |u_s|^2 
ds\right] < \infty\right\},
 \\
 \mathcal{U}^y &:= \left\{ u : \Omega \times [0,T] \rightarrow U^y\ \big|\ 
u_\cdot \text{ is }\mathbb{F}\text{-adapted},\ \mathbb{E}\left[\int_0^T |u_s|^2 
ds\right] < 
\infty\right\}.
 \end{aligned}
\end{equation}
As a demonstration of how randomness effects the tagged model, 
consider the tagged dynamics with Brownian small-scale interactions and a 
random terminal condition $y_T\in L^2_{\mathcal{F}_T}(\rd)$,
\begin{equation}
\begin{cases}
 dY_t = (u^y_t + B^y_t)dt, & \forall\ t\in[0,T],\ u^y_\cdot \in \mathcal{U}^y,
 \\
 Y_T = y_T.
 \end{cases}
\end{equation}
The naive solution $Y_t = y_T 
- \int_t^T (u^y_s + B^y_s) ds$ is not $\F_t$-adapted, it depends on $(B^y_s; 
t\leq s \leq T)$! On the other hand, 
\begin{equation}
\label{eq:conditional_expectation_Y}
Y_t = 
\x{y_T - 
\int_t^T u^y_s + B^y_s ds\ |\ 
\F_t} 
\end{equation}
is $\F_t$-adapted. If $Y_t$ from 
\eqref{eq:conditional_expectation_Y} is square-integrable for all 
$t\in[0,T]$, which it is in the current setup,
the Martingale Representation Theorem grants
existence of a unique square-integrable 
$\mathbb{R}^{d \times (w_x 
+ w_y)}$-valued and $\F_t$-adapted process $Z_\cdot$ such that
\begin{equation}
\begin{cases}
\label{eq:bsde_1st}
 Y_t - \int_0^t (u^y_s + B^y_s)ds = \int_0^t Z_sdB_s,& 
\forall t\in[0,T],
 \\
 Y_T = y_T.
\end{cases}
\end{equation} 
Equation \eqref{eq:bsde_1st} constitutes a 
BSDE. The conditional expectation \eqref{eq:conditional_expectation_Y} can be 
interpreted as the 
$L^2$-projection of the tagged's future path onto currently available 
information. Therefore, a practical interpretation of $Z_\cdot$ is that it is a 
supplementary control used by the tagged, to make its path to $y_T$ the 'best 
prediction' at any instant in time. From a modeling point 
of 
view, the tagged pedestrian thus uses two control processes:
\vspace{-0.2cm}
\begin{itemize}
 \item $u^y_\cdot$ - to heed preferences on initial position, speed, congestion 
and more. It is the tagged's subjective best response to the environment.
  \item $Z_\cdot$ - to predict the best path to $y_T$ given 
$u^y_\cdot$. It is a square-integrable process, implicitly given by the 
Martingale Representation Theorem.
\end{itemize}
\noindent
Interaction between pedestrians at time $t$ is introduced via the mean-field 
of ordinary pedestrians, $\law{X_t} := \p\circ(X_t)^{-1}$, and 
tagged pedestrians, $\law{Y_t}$. They approximate the over-all 
behavior of the crowds in the 
large population limit, under the assumption that within each of the two 
groups individuals are indistinguishable (anonymous). This assumption is in 
place throughout the paper. An example of a mean-field dependent preference 
is the following: to safely accomplish its mission, a security team prefers 
that no individual deviates too far away from the mean position of the team. 
Also, effects like congestion, the extra effort needed when 
moving in a high density area, and aversion, repulsion from other pedestrians, 
can be captured with distribution dependent coefficients. 
\vspace{0.2cm}\\
In a random environment, with mean-field interactions, we formulate the 
dynamics of representative group members (ordinary and tagged) in the model as
\begin{equation}
\label{eq:backward_dynamics}
\left\{
  \begin{aligned}
   dY_t &= b^y(t, \Theta^y_t, Z_t, \Theta^x_t)dt + Z_ 
tdB_t,
 \\
 dX_t &= b^x(t, \Theta^y_t, Z_t, \Theta^x_t)dt + 
\sigma^x(t, \theta^y_t, Z_t, \theta^x_t)dB_t^x,
  \\
  Y_T &= y_T,\qquad X_0 = x_0,
  \end{aligned}
 \right.
\end{equation}
where $\Theta^y_t := (Y_t, \law{Y_t}, u^y_t)$, $\theta^y_t := (Y_t, 
\law{Y_t})$ and
\begin{equation}
\begin{aligned}
 b^x, b^y &: \Omega \times [0,T] \times \rd \times \mathcal{P}(\rd) \times U^y 
\times \mathbb{R}^{d\times(w_x+w_y)} \times \rd \times \mathcal{P}(\rd) \times 
U^x \rightarrow \mathbb{R},
\\
\sigma^x &: \Omega \times [0,T] \times \rd \times \mathcal{P}(\rd)\times 
\mathbb{R}^{d\times(w_x+w_y)} \times \rd\times \mathcal{P}(\rd) \rightarrow 
\mathbb{R}.
\end{aligned}
\end{equation}
$\Theta^x_t$ and $\theta_t^x$ are defined correspondingly.
Distribution-dependence makes \eqref{eq:backward_dynamics} a so-called 
\emph{controlled mean-field BFSDE}. Appendix 
\ref{sec:appendixA} summarizes results on existence of
unique solutions to mean-field BFSDEs and states assumptions strong enough so 
that there exists a unique solution to \eqref{eq:backward_dynamics} given any 
feasible control pair $(u^x_\cdot, u^y_\cdot) \in \mathcal{U}^x\times 
\mathcal{U}^y$. The assumptions are in force throughout this paper.
\begin{remark}
Fundamental diagrams, that describe the marginal relations between speed, 
density and flow in a crowd, are not necessary in the construction of $b^x$ and 
$b^y$, as 
functions of crowd density. It is pointed out in \cite{cristiani2015modeling} 
that the use of fundamental diagrams in pedestrian crowd models is an artifact 
from road traffic models, without proper justification in the case of 
two-dimensional flows. Instead, the more natural (for the purpose of 
modeling pedestrian crowd motion) multidimensional velocity fields $b^x$ and 
$b^y$ are used here.
\end{remark}

\subsection{Pedestrian preferences}

Modeling pedestrian preferences is a delicate task. Not only is 
gathering of and calibration to 
empirical data difficult for many reasons, but different setups lead to vastly 
different mathematical formulations of rationality. The focus in this paper is 
on setups where pedestrian groups are controlled by a central planner. 
Other possible setups are discussed in Section~\ref{sec:conclusions}.
\vspace{0.2cm}\\
The ordinary and tagged pedestrians pay the cost $f^x$ and $f^y$ per time unit, 
respectively,
\begin{equation}
 f^x, f^y : \Omega \times [0,T] \times \rd \times \mathcal{P}(\rd) \times U^y 
\times \mathbb{R}^{d\times(w_x+w_y)} \times \rd \times \mathcal{P}(\rd) \times 
U^x \rightarrow \mathbb{R}.
\end{equation}
On top of this, ordinary pedestrians pay a terminal cost $h^x$ at time $T$, 
and tagged pedestrians pay an initial cost $h^y$ at time $0$,
\begin{equation}
 h^x, h^y : \Omega \times \rd \times \mathcal{P}(\rd) \times \rd \times 
\mathcal{P}(\rd) \rightarrow \mathbb{R}.
\end{equation}
Given a control feasible pair $(u^x_\cdot, u^y_\cdot)$ the total cost is $J^x : 
\mathcal{U}^x \rightarrow \mathbb{R}$ for the 
ordinary pedestrian, and $J^y: 
\mathcal{U}^y \rightarrow \mathbb{R}$ for the tagged,
\begin{equation}
\begin{aligned}
 J^x(u^x_\cdot; u^y_\cdot) &:= \mathbb{E}\left[\int_0^T 
f^x(t,\Theta^y,Z_t,\Theta^x_t)dt + h^x(\theta^y_T, \theta^x_T) \right],
\\
 J^y(u^y_\cdot; u^x_\cdot) &:= \mathbb{E}\left[\int_0^T 
f^y(t,\Theta^y,Z_t,\Theta^x_t)dt + h^y(\theta^y_0, \theta^x_0) \right].
\end{aligned}
\end{equation}
\subsubsection{Mean-field type game}
Consider the following situation. Within each crowd, \textit{pedestrians 
cooperate}, but on a group-level, the \textit{crowds compete}. This 
constitutes a so-called \textit{mean-field type game} between the crowds. A 
Nash equilibrium in the game is a pair of feasible controls, $(\hat{u}^x_\cdot, 
\hat{u}^y_\cdot) \in 
\mathcal{U}^x \times \mathcal{U}^y$, satisfying the inequalities
\begin{equation}
\label{eq:Nash_game}
 \begin{cases}
  J^x(u_\cdot; \hat{u}^y_\cdot) \geq J^x(\hat{u}^x_\cdot; 
\hat{u}^y_\cdot),& \forall\, u_\cdot \in \mathcal{U}^x,
  \\
  J^y(u_\cdot; \hat{u}^x_\cdot) \geq J^y(\hat{u}^y_\cdot; 
\hat{u}^x_\cdot),& \forall\, u_\cdot \in \mathcal{U}^y.
\end{cases}
\end{equation}
The next result is a Pontryagin's type stochastic maximum principle, and yields 
necessary and sufficient conditions for any 
control pair satisfying \eqref{eq:Nash_game}. A proof is provided in 
Appendix~\ref{sec:appendixProof}.
\begin{theorem}
\label{thm:game_smp}
Suppose that $(\hat{u}^x_\cdot, \hat{u}^y_\cdot)$ is a Nash 
equilibrium, i.e. satisfies~\eqref{eq:Nash_game}, and let the regularity 
assumptions of Lemma~\ref{lem:estimate1} be in force. Denote the corresponding 
state processes, given by~\eqref{eq:backward_dynamics}, by $\hat{X}_\cdot$ and 
$(\hat{Y}_\cdot, \hat{Z}_\cdot)$ respectively.  
Let $(p^{xx}_\cdot, q^{xx}_\cdot, q^{xy}_\cdot)$, $p^{xy}_\cdot$,  
$p^{yy}_\cdot$ and $(p^{yx}_\cdot, q^{yx}_\cdot, q^{yy}_\cdot)$ solve the 
adjoint equations
\begin{equation}
\label{eq:adjoint_in_thm}
 \left\{
 \begin{aligned}
  dp^{xx}_t &= -\left\{\partial_x \hat{H}^x_t + 
\mathbb{E}\left[^*(\partial_{\mu^x}\hat{H}^x_t)\right]\right\}dt + 
q^{xx}_tdB^x_t + q^{xy}_tdB^y_t,
  \\
  dp^{xy}_t &= -\left\{\partial_y \hat{H}^x_t + 
\mathbb{E}\left[^*(\partial_{\mu^y}\hat{H}^x_t\right]\right\}dt - \partial_z 
\hat{H}^x_tdB_t,
\\
dp^{yy}_t &= - \left\{ \partial_y \hat{H}^y_t + 
\mathbb{E}\left[^*(\partial_{\mu^y}\hat{H}^y_t)\right]\right\}dt - 
\partial_z\hat{H}^y_tdB_t,
\\
dp^{yx}_t &= -\left\{\partial_x\hat{H}^y_t + 
\mathbb{E}\left[^*(\partial_{\mu^x}\hat{H}^y_t)\right]\right\}dt + 
q^{yx}_tdB^x_t + q^{yy}_tdB^y_t,
\\
p^{xx}_T &= -\left\{\partial_x \hat{h}^x + 
\mathbb{E}\left[^*(\partial_{\mu^x}\hat{h}^x)\right]\right\},\quad p^{xy}_0 = 
0,
\\
p^{yy}_0 &= \partial_y \hat{h}^y + 
\mathbb{E}\left[^*(\partial_{\mu^y}\hat{h}^y)\right], \quad 
p^{yx}_T = 0,
 \end{aligned}
 \right.
\end{equation}
where $H^i$, $i\in\{x,y\}$, is the Hamiltonian, defined by
\begin{equation}
\label{eq:def_hamiltonian}
\begin{aligned}
 H^i(\omega, t, y, \mu^y, v, z, x, \mu^x, u, p^{ix}, p^{iy}, q^{ix}) := 
\sum_{j \in \{x,y\}} b^j(\omega, t, y, \mu^y, v, z, x, \mu^x, u)p^{ij}
 \\ 
 + \sigma^x(\omega,t, y, \mu^y, x, \mu^x)q^{ix}
 -f^i(\omega, t, y, \mu^y, v, z, x, \mu^x, u).
\end{aligned}
\end{equation}
Then
\begin{equation}
\begin{aligned}
 \label{eq:necessary_game_eq}
\hat{u}^x_t 
&= 
\underset{v\in U^x}{\text{argmax}}\ H^x(t, 
\hat{\Theta}^y_t, \hat{Z}_t, \hat{\theta}^x_t,
v, p^{xx}_t, p^{xy}_t, q^{xx}_t), \quad a.e.\text{-}t,\ \mathbb{P}\text{-}a.s.
  \\
  \hat{u}^y_t 
  &=
  \underset{v\in U^y}{\text{argmax}}\ H^y(t, 
\hat{\theta}^y_t, v, \hat{Z}_t, \hat{\Theta}^x_t, p^{yx}_t, p^{yy}_t, 
q^{yx}_t), \quad a.e.\text{-}t,\ \mathbb{P}\text{-}a.s.
\end{aligned}
\end{equation}
If furthermore $H^i$ is concave 
in $(y,\mu^y, v, z, x, \mu^x, u)$, and $h^i$ is convex in 
$(y,\mu^y, x, \mu^x, z)$, $i\in\{x,y\}$, then any feasible control pair 
satisfying~\eqref{eq:necessary_game_eq} is a Nash equilibrium in the 
mean-field type game.
\end{theorem}
\begin{remark}
 Note that in the mean-field type game, the tagged can be thought of as a 
\textit{major 
player}, influencing the ordinary crowd. Furthermore, the central 
planner has access to a model of the ordinary crowd and this is necessary for 
the determination of the equilibrium control.
\end{remark}

\subsubsection{Optimal control of mean-field type}
If the central planner does not have access to a model of the crowd of 
ordinary 
pedestrians, any interaction with the surrounding environment will then 
enter the tagged model as a random signal. This is covered by 
the $(\omega,t)$-dependence of $b^y$ and $f^y$. The mean-field type game 
then reduces to a so-called \textit{optimal control problem of mean-field type},
\begin{equation}
\label{eq:backward_CP}
\left\{
 \begin{aligned}
  \underset{u^y_\cdot \in\mathcal{U}^y}{\text{min}}&\ \mathbb{E}\left[\int_0^T 
f^y(t, \Theta^y_t, Z_t)dt + h^y(\theta^y_0, Z_0)\right],
  \\
  \text{s.t.}&\ \ dY_t = b^y(t, \Theta^y_t, Z_t)dt + Z_ tdB_t,
  \\
  &\ \ Y_T = y_T.
  \end{aligned}
 \right.
\end{equation}
Problem \eqref{eq:backward_CP} is a special case of 
\eqref{eq:Nash_game} and necessary and sufficient conditions follow as a 
corollary to Theorem~\ref{thm:game_smp}.
\begin{corollary}
\label{thm:nec_back_smp}
Suppose that $\hat{u}^y_\cdot$ solves~\eqref{eq:backward_CP} 
and denote the corresponding tagged state $(\hat{Y}_\cdot, \hat{Z}_\cdot)$. Let 
$p_\cdot$ solve the adjoint equation
\begin{equation}
\label{eq:adjoint_in_corr}
 \left\{
 \begin{aligned}
  dp_t &= -\left\{\partial_y\mathcal{H}(t, \hat{\Theta}^y_t, \hat{Z}_t, p_t) 
+ \mathbb{E}\left[^*(\partial_\mu 
\mathcal{H}(t, \hat{\Theta}^y_t, \hat{Z}_t, p_t)\right]\right\}dt
\\
& \quad
- \partial_z \mathcal{H}(t, \hat{\Theta}^y_t, \hat{Z}_t, p_t) dB_t,
\\
p_0 &= \partial_yh^y(\hat{\theta}^y_0) + \mathbb{E}\left[^*(\partial_\mu 
h^y(\hat{\theta}^y_0)\right],
 \end{aligned}
 \right.
\end{equation}
where $\mathcal{H}$ is the Hamiltonian
\begin{equation}
 \mathcal{H}(\omega,t,y,\mu,v,z,p) := b^y(\omega,t,y,\mu,v,z)p - 
f^y(t,y,\mu,v,z).
\end{equation}
Then
\begin{equation}
\label{eq:NBSMP_eq}
\hat{u}^y_t = \underset{v\in U^y}{\text{argmax}}\ \mathcal{H}(t, 
\hat{\theta}^y_t, v, \hat{Z}_t, p_t), \quad a.e.\text{-}t,\ 
\mathbb{P}\text{-}a.s.
\end{equation}
If moreover $\mathcal{H}$ is concave in $(y,\mu, v, z)$ and that $h$ is 
convex in 
$(y,\mu)$, then any feasible $\hat{u}^y_\cdot$ that satisfies 
\eqref{eq:NBSMP_eq}, almost surely for 
a.e. 
$t$, is an optimal control to~\eqref{eq:backward_CP}.
\end{corollary}

\section{Simulations}
\label{sec:distributed}

This section is devoted to numerical simulations of the tagged model 
under various external influences and internal preferences. First, two 
scenarios in the optimal control version of the model are considered, including 
preferences on velocity, avoidance and interaction via the mean position of the 
group. Secondly, asymmetric bidirectional flow is simulated with the 
full mean-field type game version of the model. None of 
the parameters used in the simulations stem from real world measurements, but 
velocity profiles  and the asymmetric bidirectional flow are compared 
qualitatively with the experimental studies
\cite{moussaid2009experimental} and \cite{zhang2012ordering}. 

\subsection{Optimal control: keeping a tagged group together}
\label{sec:keeptogether}

In this scenario, the common goal of the tagged pedestrians is to stay close to 
the group mean, while conserving energy and initiating in the proximity of 
$y_0\in\mathbb{R}^2$. Distance to the group mean is one of the simplest 
mean-field effects that can be considered. Nonetheless it is a 
distribution-dependent quantity, and the control problem characterizing tagged 
pedestrian behavior in this scenario is the nonstandard optimization problem 
\eqref{eq:test0}. The components of \eqref{eq:test0} are presented in 
Table~\ref{table:keep_together}. There is no surrounding crowd present, only 
tagged pedestrians occupy the space. 
\begin{equation}
\label{eq:test0}
\left\{
 \begin{aligned}
  \underset{u_\cdot \in\mathcal{U}}{\text{min}}&\ \frac{1}{2}\x{\int_0^T \left(
\lambda_{\text{cont}}\left(u_t^y\right)^2 + \lambda_{\text{attr}}\left(Y_t - 
\mathbb{E}[Y_t]\right)^2 \right) dt + 
\lambda_{\text{init}}(Y_0 - 
y_0)^2},
  \\
  \text{s.t.}&\ \ dY_t = (u_t^y + \lambda_{\text{noise}}B_t)dt + Z_ tdB_t,
  \\
  &\ \ Y_T = y_T.
  \end{aligned}
 \right.
\end{equation}
\begin{center}
\begin{table}[h]
\ra{1.3}
\begin{tabular}{@{}llllrrrr@{}}
\toprule
\textbf{Velocity component} & \textbf{Form}
\\ 
\midrule
Internal velocity (control)  & $u_\cdot^y \in \mathcal{U}^y$
\\
Acceleration noise  & $\lambda_{\text{noise}} B_\cdot^y$
\vspace{0.2cm} \\
\toprule
\textbf{Preference (penalty)} & \textbf{Form}
\\ 
\midrule
Energy usage per unit time & $\lambda_{\text{cont}} \left(u_\cdot^y\right)^2$
\\
Distance from group mean per unit time & $\lambda_{\text{attr}}
\displaystyle\left(Y_\cdot - \mathbb{E}[Y_\cdot]\right)^2$
\\
Distance from $y_0\in \mathbb{R}^2$ at $t=0$ & $\lambda_{\text{init}}(Y_0 - 
y_0)^2$
\\
\bottomrule
\end{tabular}
\vspace{0.2cm}\\
\caption{Keeping a tagged group together: components}
\label{table:keep_together}
\end{table}
\end{center}
\vspace{-.8cm}
The scenario is simulated for two sets of parameters, 
see Table~\ref{table:2paraSet}. The mean-field BFSDE system of 
equations characterizing optimal behavior, given by 
Corollary~\ref{thm:nec_back_smp}, is solved with the least-square Monte Carlo 
method of \cite{bender2008time}. The result is presented in Figure~\ref{fig:1}. 
The group walks approximately on the straight line from the starting area to 
the target point. Remember that the initial position of the tagged is chosen 
rationally by solving \eqref{eq:test0}. There is a trade-off between starting 
close to $y_0$ and walking with high speed, and the groups rationally initiates 
not at $y_0$, but somewhere between $y_0$ and $y_T$. The group that prefers 
proximity to other group members does indeed move in a more compact formation. 
The difference appears clearly when looking at the mean distance-to-mean of the 
tagged group, see Figure~\ref{fig:dist-to-mean}.
\begin{center}
\begin{table}[h]
\ra{1.3}
\begin{tabular}{@{}lllllllll@{}}
  & $\lambda_{\text{noise}}$ & $\lambda_{\text{cont}}$ & 
$\lambda_{\text{attr}}$ & $\lambda_{\text{init}}$ & $\mathbb{E}[y_0]$ & $y_T$ & 
$T$
  \\ 
  \midrule
  \textbf{Set 1 (with distance-to-mean penalty)} & 1 & 50 & 50 & 
10 & 
[0.1,0.1] & [2,2] & 1
  \\ 
  \textbf{Set 2 (without distance-to-mean penalty)} & 1 & 50 & 0 & 
10 & 
[0.1,0.1] & [2,2] & 1
  \\
\bottomrule
\end{tabular}
\vspace{0.2cm}\\
\caption{Keeping a tagged group together: parameter values. The preferred 
initial position is normally distributed around $[0.1, 0.1]$.}
\vspace{-0.8cm}
\label{table:2paraSet}
\end{table}
\end{center}
\begin{figure}[h!]
\label{fig:m2m_density}
  \label{fig:simulation}
 \includegraphics[scale = 0.7, trim = 2cm 0cm 0cm 0cm]{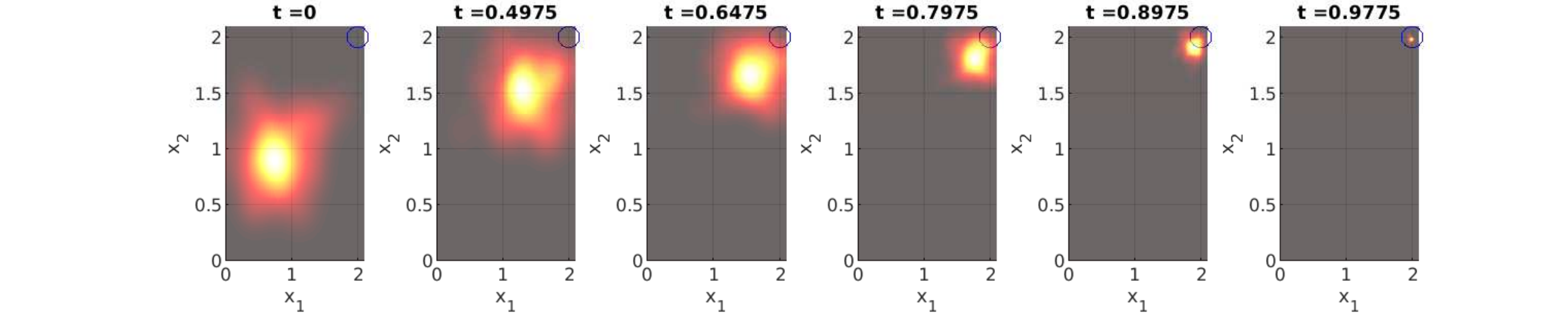}
 \includegraphics[scale = 0.7, trim = 2cm 0cm 0cm 0cm]{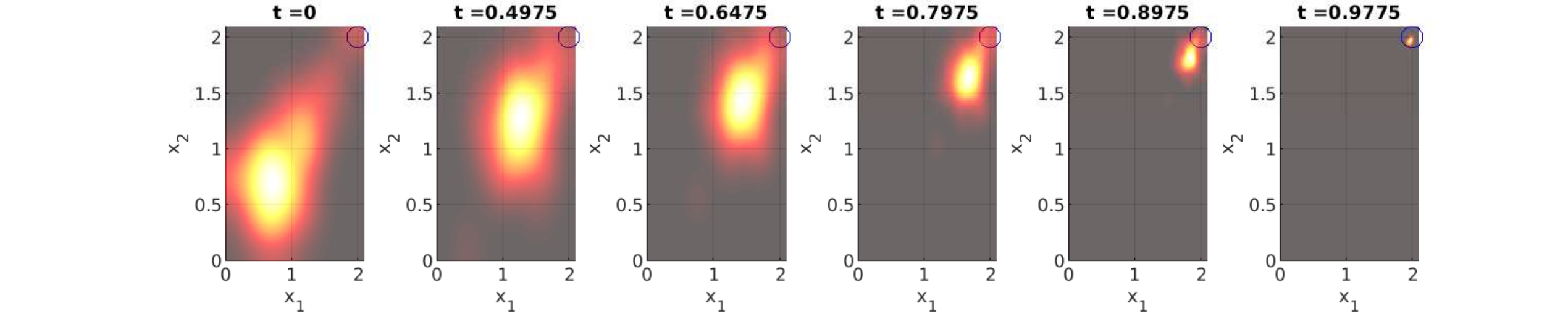}
 \caption{\textup{Top row:} Crowd density evolution when $\lambda_2 = 50$. 
\textup{Bottom row:} Crowd density evolution when $\lambda_2 = 0$.}
\label{fig:1}
\end{figure}
\begin{figure}[h!]
  \label{fig:m2m_distance}
  \includegraphics[scale = 0.45]{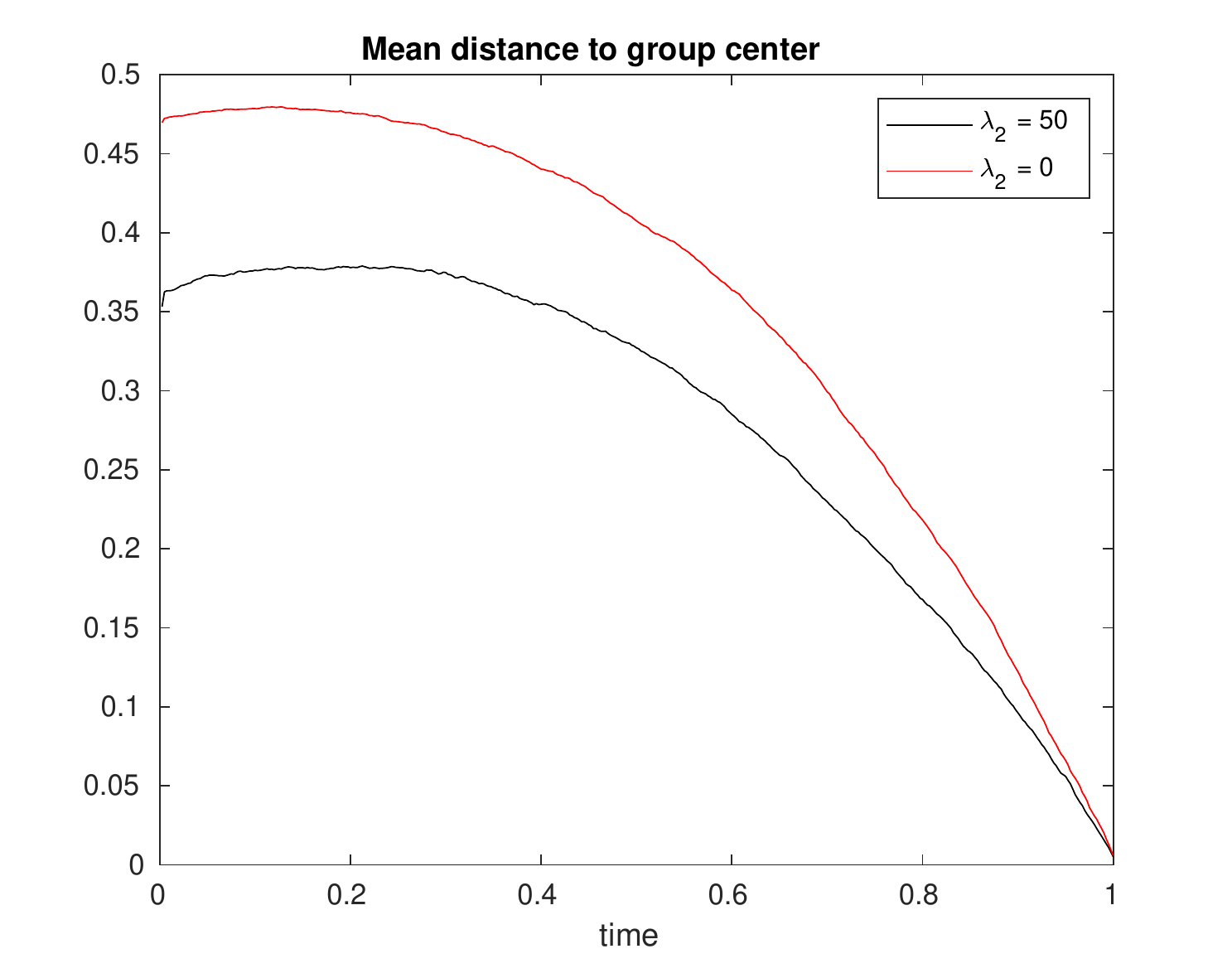}
 \caption{A tagged group acting under preferences with a 
distance-to-mean penalty indeed moves in a more compact formation than a group 
that does not.}
  \label{fig:dist-to-mean}
\end{figure}

\subsection{Optimal control: desired velocity}

Linear-quadratic scenarios have accessible closed form solutions by the method 
of matching. We want to mention the case where the tagged's goal is to move at 
its desired velocity $v_{\text{des}}$, similar to what was originally 
introduced as \textit{the desired speed and direction} in 
\cite{helbing1995social}. This is an important special case, since desired 
velocity is measurable in live experiments. See for example
\cite{moussaid2009experimental} for the speed profile of pedestrian walking in 
a straight corridor, starting from standing still. 
The scenario is formulated as standard optimal control problem, 
\eqref{eq:test1}, and 
the setting is summarized in Table~\ref{table:desvel}. 
\begin{equation}
\label{eq:test1}
\left\{
 \begin{aligned}
  \underset{u \in\mathcal{U}[0,T]}{\text{min}}&\ \frac{1}{2}\x{\int_0^T \left(
\lambda_{\text{cont}}\left(u_t^y\right)^2 + \lambda_{\text{des}}\left(u_t^y - 
v_{\text{des}}(t)\right)^2 + \lambda_{\text{rep}}(Y_t - 
Q)^2\right) dt},
  \\
  \text{s.t.}&\ \ dY_t = (u^y_t + \lambda_{\text{noise}}B^y_t)dt + Z_ tdB_t,
  \\
  &\ \ Y_T = y_T.
  \end{aligned}
 \right.
\end{equation}
\begin{table}[h]
\centering
\ra{1}
\begin{tabular}{@{}llllrrrr@{}}
\toprule
\textbf{Velocity component} & \textbf{Form}
\\ 
\midrule
Internal velocity (control)  & $u^y_\cdot \in \mathcal{U}^y$
\\
Acceleration noise  & $\lambda_{\text{noise}} B_\cdot$
\vspace{0.2cm} \\
\toprule
\textbf{Preference} & \textbf{Form}
\\ 
\midrule
Energy usage per unit time & $\lambda_{\text{cont}} \left(u^y_\cdot\right)^2$
\\
Uncomfortable velocity per unit time & $\lambda_{\text{des}}\left(u^y_\cdot - 
v_{\text{des}}(t)\right)^2$
\\
Distance from $Q \in \mathbb{R}^2$ per unit time & 
$\lambda_{\text{rep}}(Y_\cdot - Q)^2$
\\
\bottomrule
\end{tabular}
\vspace{0.2cm}\\
\caption{Desired velocity: components}
\label{table:desvel}
\end{table}\\ \noindent
In view of Corollary \ref{thm:nec_back_smp}, the optimal control is
\begin{equation}
\hat{u}^y_t = \frac{p_t + 
\lambda_{\text{des}}v_{\text{des}}(t)}{\lambda_{\text{cont}} + 
\lambda_{\text{rep}}},
\end{equation}
With the ansatz $\hat{Y}_t = \gamma(t)p_t + \eta(t)B^y_t + \theta_t$, 
$\gamma(T) 
= \eta(T) = 0$ and $\theta(T) = y_T$, a matching argument gives the optimally 
controlled dynamics up to a system of ODEs:
\begin{equation}
 \begin{cases}
  \frac{d}{dt}\gamma(t) = -\lambda_{\text{rep}}\gamma(t)^2 + 
\frac{1}{\lambda_{\text{cont}} + 
\lambda_{\text{des}}}, &\gamma(T) = 0,
  \\
  \frac{d}{dt}\eta(t) = - \lambda_{\text{rep}}\gamma(t)\eta(t) + 
\lambda_{\text{noise}}, &\eta(T) 
= 0,
  \\
  \frac{d}{dt}\theta(t) = -\lambda_{\text{rep}}\gamma(t)(\theta(t) - Q) + 
\frac{\lambda_{\text{des}}v_{\text{
des}}(t)}{\lambda_{\text{cont}}+\lambda_{\text{des}}} , 
&\theta(T) = 
y_T,
 \\
 \hat{Z}_t = \eta(t).
 \end{cases}
\end{equation}
In Figure~3, the simulated tagged crowd density is presented for two 
values of 
$Q$. The desired velocity is set to be negative in both directions for 
$t\in[0,T/2]$, and positive for $t \in(T/2, T]$, which corresponds to a 
preference to first move south-west during the first half of the time period 
and then turn around and move north-east. The parameter $\lambda_{\text{rep}}$ 
is set to a negative value, hence the tagged prefers to avoid $Q$. Parameters 
used in the simulation are summarized in Table~\ref{table:desvel_para}.
The trade-off between walking in the desired velocity and walking close to 
the diamond $Q$ before reaching the target circle is clearly visible.  Recall 
that the initial position is determined by the optimization procedure. In 
this scenario, there is no preference on initial position and the 
tagged group compensates for the location of $Q$ by changing its initial 
position!
\vspace{0.2cm}\\
In \cite{moussaid2009experimental} the average time-dependent 
velocity of a pedestrian initially standing still is measured experimentally in 
the absence of interactions. The result is a relationship between speed and 
time, than can be used as data for $v_{\text{des}}$. Approximating the graph 
presented in \cite{moussaid2009experimental} with 
\begin{equation}
\label{eq:vdesss}
 v_{\text{des}}(t) = \max\{0.1, \text{arctan}(\pi t - 1.6)\},
\end{equation}
the scenario is simulated with the two presented in 
Table~\ref{table:desvel_para_2}. The result is presented in 
Figure~4. The parameter set with a higher penalty on from deviation from 
desired velocity naturally results in a velocity profile closer to 
$v_{\text{des}}(t)$.
\begin{center}
\begin{table}[h!]
\ra{1.3}
\begin{tabular}{@{}lllllllllll@{}}
  & $\lambda_{\text{noise}}$ & $\lambda_{\text{cont}}$ & $\lambda_{\text{des}}$ 
& $\lambda_{\text{rep}}$ & 
$Q$ & $v_{\text{des}}(t)$ & $y_T$ & $T$
  \\ 
  \midrule
  \textbf{Set 1} & 0.5  & 0.5 & 1 & -2 & 
$[-0.5,-0.5]$ & $\text{sign}\left(t-\frac{T}{2}\right)[3,3]$ & $[0.1,0.1]$ & 1
  \\
  \textbf{Set 2} & 0.5  & 0.5 & 1 & -2 & 
$[1.5,1.5]$ & $\text{sign}\left(t-\frac{T}{2}\right)[3,3]$ & $[0.1,0.1]$ & 1
  \\
\bottomrule
\end{tabular}
\vspace{0.2cm}\\
\caption{Desired velocity: parameter values.}
\label{table:desvel_para}
\end{table}
\end{center}
\vspace{-1cm}
\begin{figure}[h!]
\hspace*{-1.5cm}                                                          
 \includegraphics[scale = 0.7, trim = 0cm 0cm 0cm 0cm]{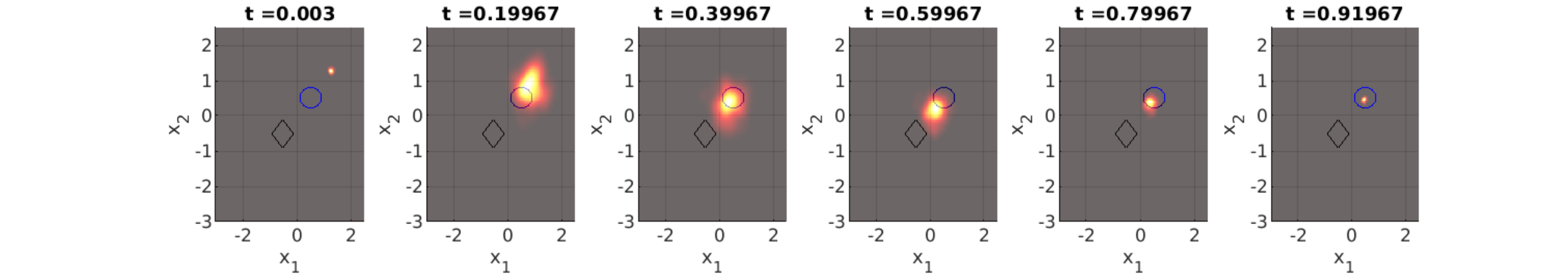} 
 \hspace*{-1.5cm}                                                          
  \includegraphics[scale = 0.7, trim = 0cm 0cm 0cm 0cm]{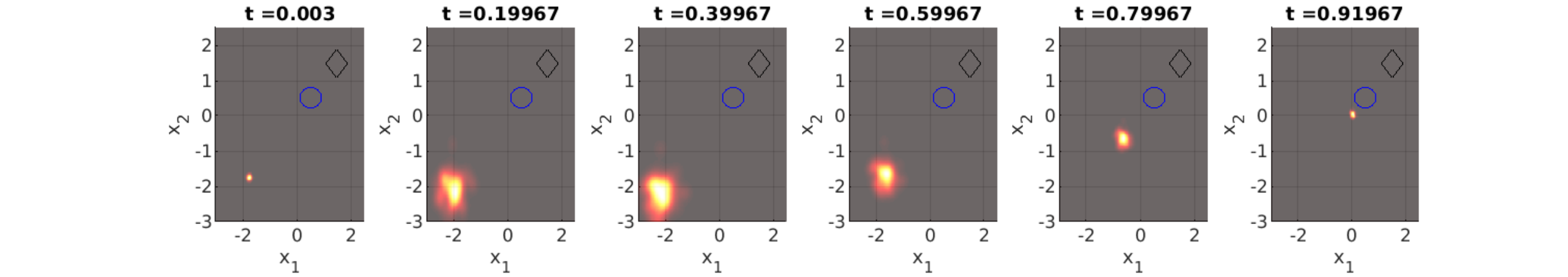}
 \label{fig:des}
 \caption{\emph{Top row:} Desired velocity with Set 1 parameters from 
Table~\ref{table:desvel_para}. \emph{Bottom row:} Desired velocity with Set 2 
parameters from 
Table~\ref{table:desvel_para}. The tagged crowd
moves first south-west and then north-east, following its desired velocity,
while avoiding the diamond on its way to the circle.}
\end{figure}
\begin{center}
\begin{table}[h!]
\ra{1.3}
\begin{tabular}{@{}lllllllllll@{}}
  & $\lambda_{\text{noise}}$ & $\lambda_{\text{cont}}$ & $\lambda_{\text{des}}$ 
& $\lambda_{\text{rep}}$ & 
$Q$ & $v_{\text{des}}(t)$ & $y_T$ & $T$
  \\ 
  \midrule
  \textbf{Set 3} & 0.1  & 0.5 & 2 & 0 & 
$[0,0]$ & Eq. \eqref{eq:vdesss} & $[0,0]$ & 4
  \\
  \textbf{Set 4} & 0.1  & 0.5 & 10 & 0 & 
$[0,0]$ & Eq. \eqref{eq:vdesss} & $[0,0]$ & 4
  \\
\bottomrule
\end{tabular}
\vspace{0.2cm}\\
\caption{Desired velocity: parameter values.}
\label{table:desvel_para_2}
\end{table}
\end{center}
\vspace{-1cm}
\begin{figure}                                                       
\begin{minipage}{0.49\linewidth}
\hspace*{1cm}                                                          
\includegraphics[scale = 0.5, trim = 2cm 0cm 0cm 0cm]{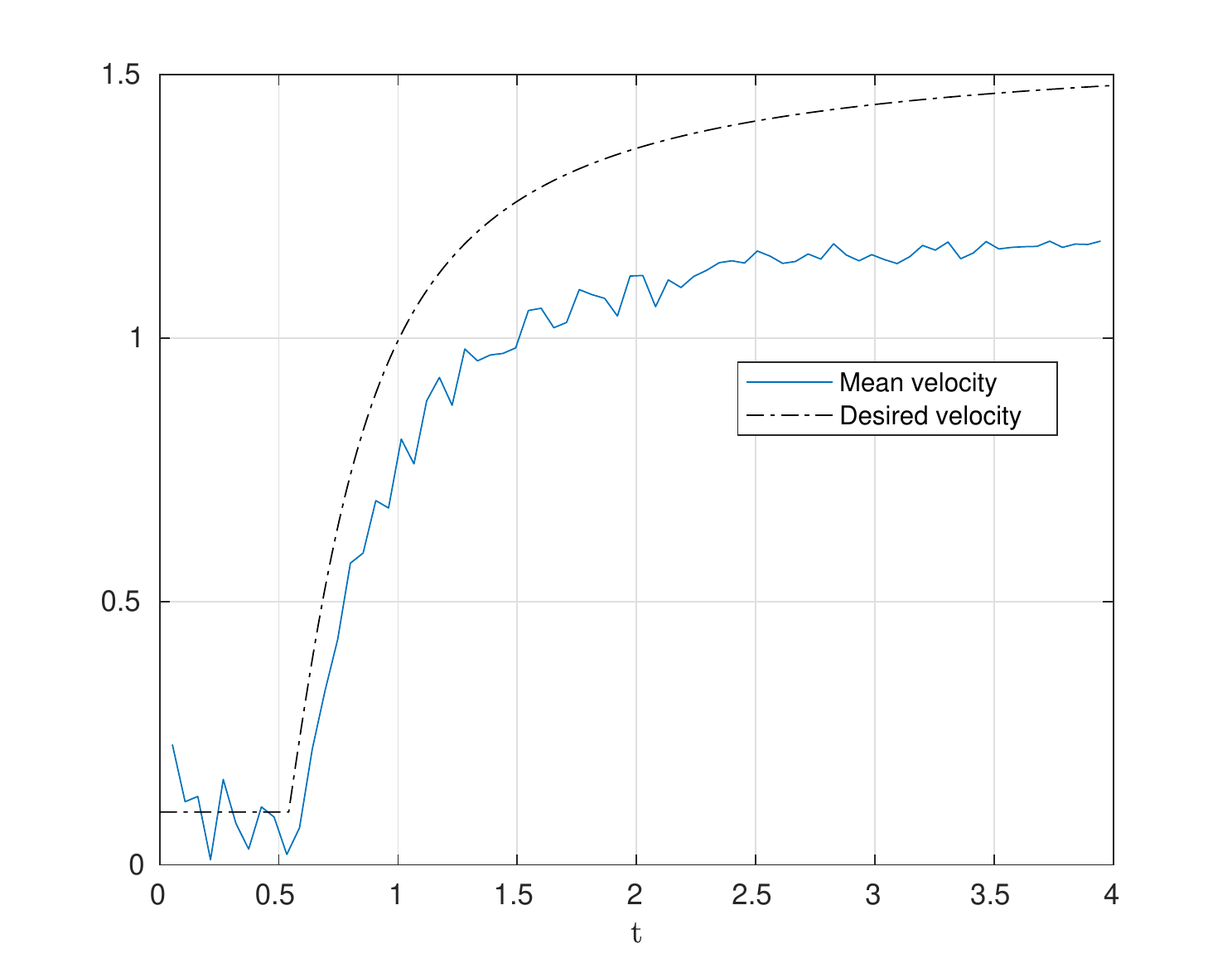}
\end{minipage}
\hfill
\begin{minipage}{0.49\linewidth}
\hspace*{.9cm}                                                          
\includegraphics[scale = 0.5, trim = 2cm 0cm 0cm 0cm]{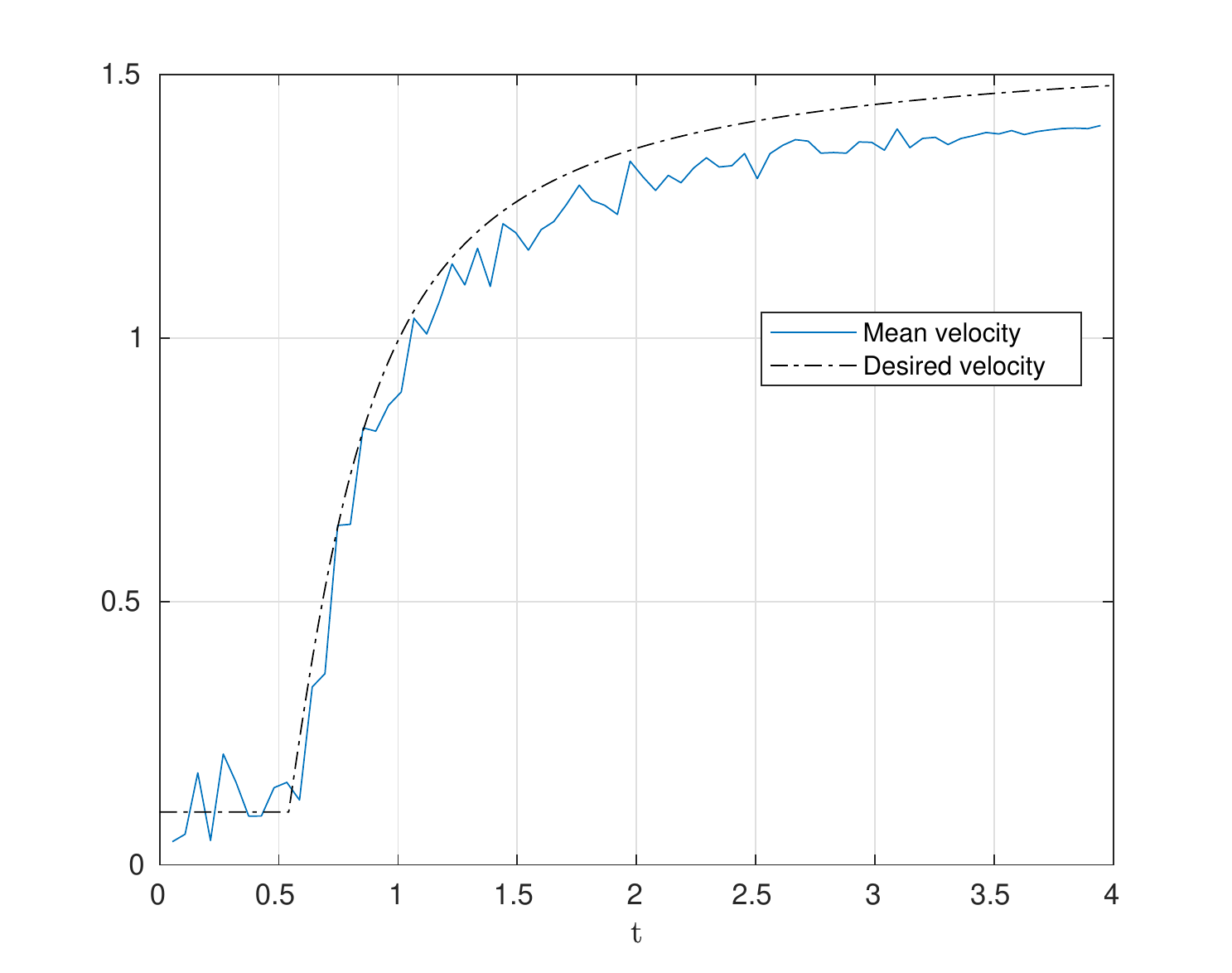}
\end{minipage}
\label{fig:desvol_2}
\caption{\emph{Left:} Desired velocity with Set 3 parameters from 
Table~\ref{table:desvel_para_2}. \text{Right:} Desired velocity with Set 4 from 
Table~\ref{table:desvel_para_2}. The mean velocity measured 
in the scenario 
\eqref{eq:test1} compared to the desired velocity \eqref{eq:vdesss}.}
\end{figure}
\subsection{Mean-field type game: asymmetric bidirectional flow}

Consider now a scenario where ordinary pedestrians initiate at $x_0\in 
\mathbb{R}^2$, close $y_T\in\mathbb{R}^2$, the location of an incident. They 
begin to walk towards the safe spot $x_T\in\mathbb{R}^2$. A tagged pedestrian is 
to end up at the location of the incident $y_T$ at time $t=T$. The tagged 
pedestrian is repelled by the mean of the ordinary pedestrian crowd, while the 
ordinary pedestrian crowd is repelled by the tagged pedestrian. This scenario 
is implemented as the MFTG \eqref{eq:asymmetric}, summarized in 
Table~\ref{table:assym}.
\begin{equation}
\label{eq:asymmetric}
 \begin{cases}
  J^x(u^x_\cdot; u^y_\cdot) = \frac{1}{2}\mathbb{E}\left[\int_0^T \left(
\lambda^x_{\text{cont}}\left(u^x_t\right)^2 + 
\lambda^x_{\text{rep}}(X_t-Y_t)^2\right)dt + 
\lambda^x_{\text{term}}(X_T-x_T)^2  \right],
  \\
  dX_t = u_t^xdt + \sigma dB^x_t,\quad \sigma\in\mathbb{R},\quad X_0 = x_0,
  \\
  J^y(u^y_\cdot; u^x_\cdot) = \frac{1}{2}\mathbb{E}\left[\int_0^T 
\left(\lambda^y_{\text{cont}}\left(u^y_t\right)^2 + \lambda^y_{\text{red}}(Y_t 
- \mathbb{E}[X_t])^2\right)dt + \lambda^y_{\text{init}}(Y_0-y_0)^2 \right],
  \\
  dY_t = \left(u_t^y + \lambda^y_{\text{noise}}dB^y_t\right) + Z_t dB_t.
 \end{cases}
\end{equation}
\begin{table}[h]
\centering
\ra{1}
\begin{tabular}{@{}llllrrrr@{}}
\toprule
ORDINARY PEDESTRIAN &  & TAGGED PEDESTRIAN & 
\\ 
\midrule
\textbf{Velocity component} & \textbf{Form} & \textbf{Velocity component} & 
\textbf{Form}
\\ 
\midrule
Internal velocity (control)  & $u^x_\cdot \in \mathcal{U}^x$ & Internal velocity 
(control) & $u^y_\cdot \in \mathcal{U}^y$
\\~\\
Diffusion noise level 		    & $\sigma$ &  Acceleration noise & 
$\lambda_{\text{noise}}^yB^y_\cdot$
\vspace{0.2cm} \\
\toprule
\textbf{Preference} & \textbf{Form} & \textbf{Preference} & \textbf{Form}
\\ 
\midrule
Energy usage per unit time & $\lambda^x_{\text{cont}} \left(  
u^x_\cdot\right)^2$ & Energy usage per unit time & $\lambda^y_{\text{cont}} 
\left(  
u^y_\cdot\right)^2$
\\~\\
Repulsion per unit time& $\lambda^x_{\text{rep}}(X_\cdot - Y_\cdot)^2$ & 
Repulsion per unit time& $\lambda^y_{\text{rep}}(Y_\cdot - 
\mathbb{E}[X_\cdot])^2$
\\~\\
Proximity to $x_T$ at $t=T$ & $\lambda_{\text{term}}^x(X_T - 
x_T)^2$ & Proximity to $y_0$ at $t=0$ & $\lambda_{\text{init}}^y(Y_0 - 
y_0)^2$
\\
\bottomrule
\end{tabular}
\vspace{0.2cm}\\
\caption{Asymmetric bidirectional flow: components}
\label{table:assym}
\end{table}
\\
In Figure~5 and Figure~6 the scenario is simulated for the parameter sets 
presented in Table~\ref{table:assym_para}. In Figure~5, the ordinary and the 
tagged do not have to cross paths to go from their initial to their 
terminal positions. The simulated paths (top plot of Figure~5) are similar in 
shape to both the outcome of the corridor experiment under 'condition 3' (no 
obstacle) of \cite{moussaid2009experimental} and the BFR-SSL experiment of 
\cite{zhang2012ordering}. These experimental studies were conducted in a 
controlled environment that is outside the tagged model presented 
in this paper. Anyhow, the tagged model replicates the separation of lanes in a 
bidirectional pedestrian flow and the uncertainty that pedestrian motion 
exhibits. In the 
density snapshots (bottom row of Figure~5) reveal that in simulated scenario, 
the tagged tagged moves in almost constant velocity towards $y_T$, while the 
ordinary group lingers a while at $x_0$ before it starts to move towards $x_T$.
In Figure~6, the tagged's and the ordinary's straight path from initiate 
position to target cross each other. In this scenario, the ordinary pedestrians 
resolve this by taking walking in a half-circle around the tagged, before 
moving towards their preferred terminal position $x_T$.

\begin{center}
\begin{table}[h]
\ra{1.3}
\begin{tabular}{@{}lllllllll@{}}
 \toprule
  TAGGED PEDESTRIAN & $\lambda^y_{\text{noise}}$ & $\lambda^y_{\text{cont}}$ & 
$\lambda^y_{\text{rep}}$ & $\lambda^y_{\text{init}}$ & $y_0$ & 
$\mathbb{E}[y_T]$ & $T$
  \\ 
  \midrule
  \textbf{Bidirectional flow} & 0.7 & 1 & -2 & 3 & 
[0,1] & [10,0] & 1
  \\ 
  \textbf{Twist} & 0.7 & 1 & -1 & 3 & 
[0,-3] & [10,-1] & 1
  \\
\midrule
ORDINARY PEDESTRIAN & $\sigma$ & $\lambda^x_{\text{cont}}$ & 
$\lambda^x_{\text{rep}}$ & $\lambda^x_{\text{end}}$ & $\mathbb{E}[x_0]$ & $x_T$ 
& $T$
\\ 
  \midrule
  \textbf{Bidirectional flow} & 0.7 & 1 & -1 & 10 & 
[10,-1]  & [0,0] & 1
\\ 
  \midrule
  \textbf{Twist} & 0.7 & 1 & -1.7 & 10 & 
[10,-2] & [0,0] & 1
\\
\bottomrule
\end{tabular}
\vspace{0.2cm}\\
\caption{Asymmetric bidirectional flow: parameters values. The initial and 
terminal constraints are normally distribution with mean tabled above, and 
standard deviation $0.3$ and $0.1$ for the tagged and ordinary pedestrian, 
respectively.}
\label{table:assym_para}
\end{table}
\end{center}
\begin{figure}
\label{fig:bidir}
 \includegraphics[scale = 0.3, trim = 2cm 1cm 0cm 
1.5cm]{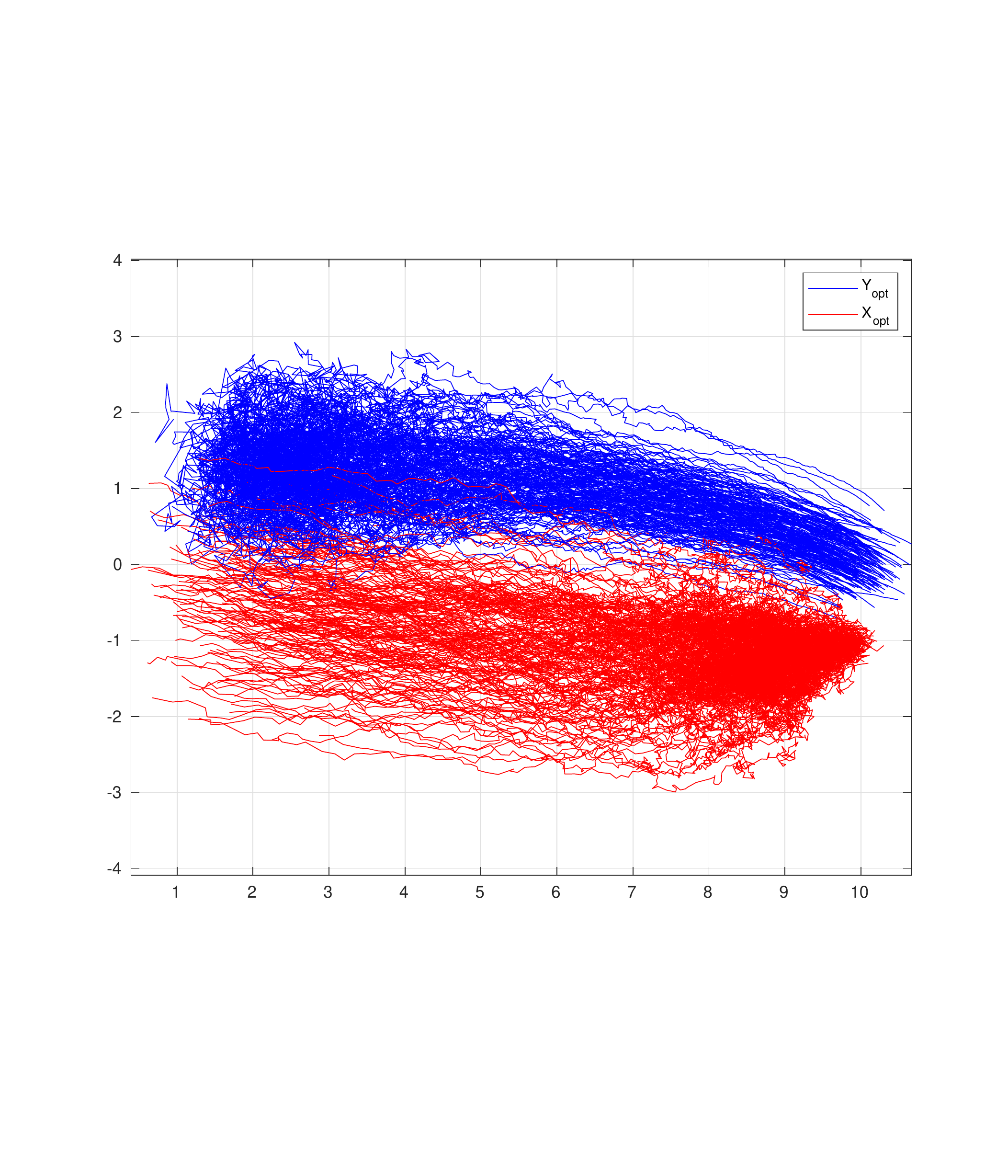}
 \includegraphics[width=1.1\textwidth, trim = 4cm 3.5cm 0cm 
3cm]{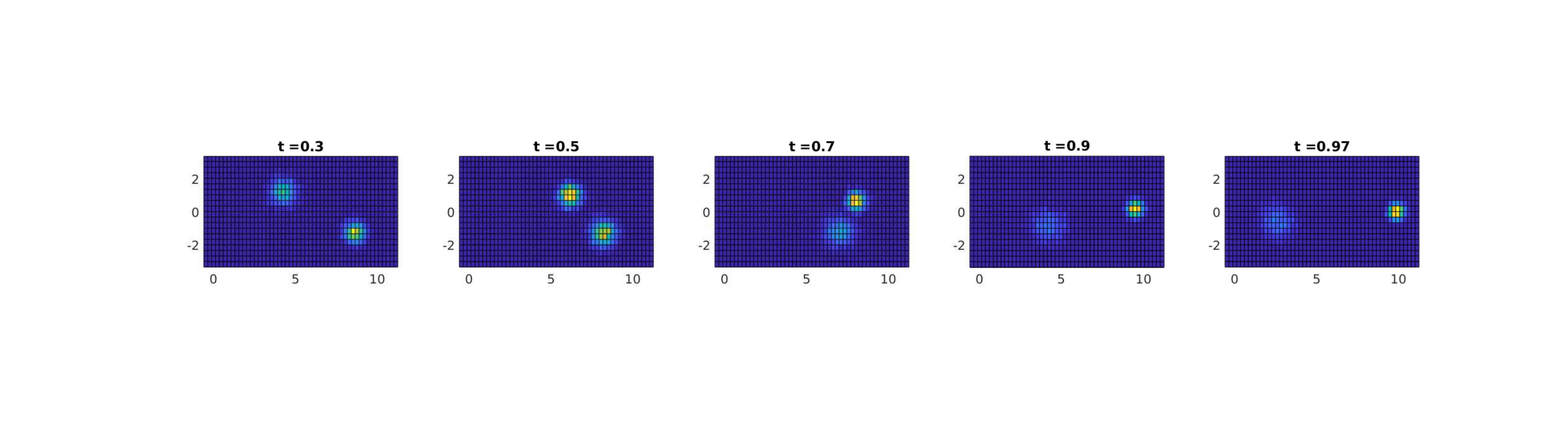} 
 \caption{\textit{\emph{Top row:} Pedestrian paths for parameter set 
'Bidirectional flow' 
from Table~\ref{table:assym_para}. \emph{Bottom row:} 
 Pedestrian density for parameter set 'Bidirectional flow' from 
Table~\ref{table:assym_para}.  }}
\end{figure}
\begin{figure}
\label{fig:bidir_twist}
\centering
\begin{minipage}{0.49\linewidth}
\hspace*{0.5cm}
\includegraphics[scale = 0.25, trim = 0cm 0cm 0cm 
0cm]{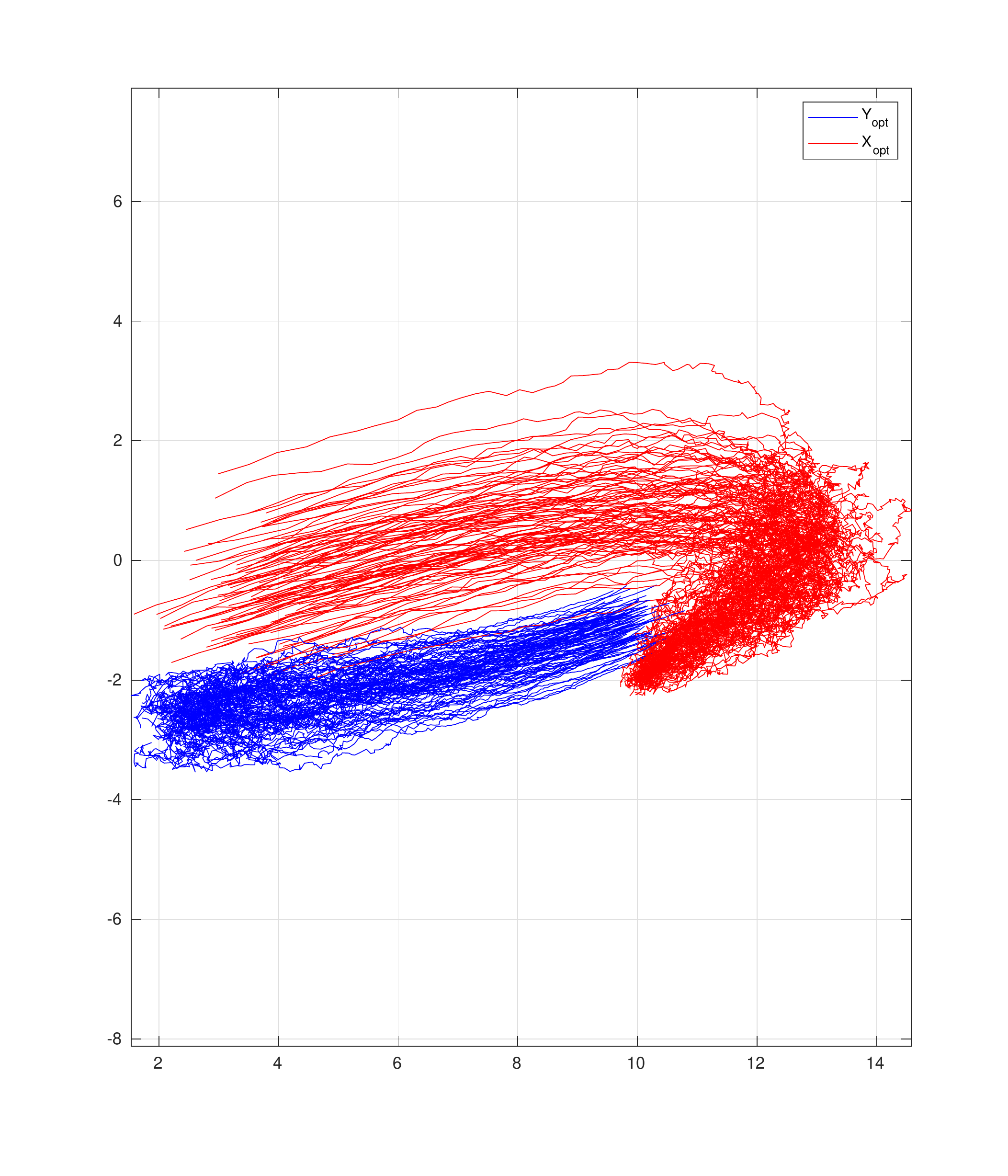}
\end{minipage}
\vspace{0cm}\\
\begin{minipage}{1\linewidth}
\includegraphics[scale = 0.5, trim = 6cm 0cm 0cm 
0cm]{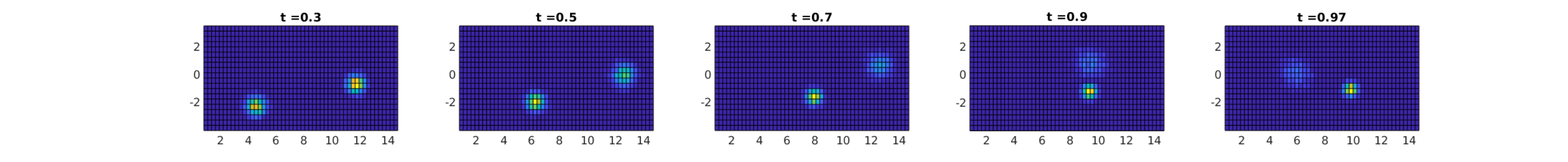}
\end{minipage}
\vspace{0cm}
\label{fig:desvol_2}
\caption{\textit{\emph{Top row:} Pedestrian paths with parameter set 'Twist' 
from Table~\ref{table:assym_para}. \emph{Bottom row:} 
 Pedestrian density with parameter set 'Twist' from 
Table~\ref{table:assym_para}.  }}
\end{figure}

\section{Concluding remarks and research perspectives}
\label{sec:conclusions}
A mean-field type game model for so-called tagged pedestrian motion has been 
presented and the reliability of the model has been studied through 
simulations. To 
perform simulations, necessary and sufficient conditions for 
a Nash equilibrium are provided in Theorem~\ref{thm:game_smp}. The theorem is 
proven under quite restrictive conditions on involved coefficient functions. 
However, necessary conditions for a Nash equilibrium in similar games are 
available under less restrictive conditions and since our proof follows a 
standard path, the conditions can certainly be relaxed. The model captures both 
game-like and minor agent-type scenarios. In the 
latter, the tagged cannot effect crowd movement while in the 
former, the tagged and the surrounding crowd have conflicting interests, 
interact, and compete. The rational pedestrian behavior in the competative 
game-like scenario is to use an equilibrium strategy.
\vspace{0.2cm}\\
There are many variations to the mean-field type game approach. The scenarios 
that have been considered in this paper fall into two rather extreme 
categories; our pedestrians have acted under either \textit{basic} or 
\textit{optimal rationality}, in the terms of \cite{cristiani2011multiscale}. 
When pedestrians neither have access to information about their 
surroundings or the ability to anticipate pedestrian behavior, the best they 
can do is to implement a control policy based on their own 
position and target position. This is a basic level of rationality. If the 
full model is available and pedestrians cooperate, they can implement a 
control policy of optimal rationality. If the tagged can 
observe crowd densities at each instant in time \textit{but not anticipate 
future movement}, dynamic pedestrian preference may be modeled as a set of 
control problems: for each 
$\tau\in[0,T]$,
\begin{equation}
\label{eq:backward_CP_rational}
\left\{
 \begin{aligned}
  \underset{u^y_.(\tau) \in\mathcal{U}^y_\tau}{\text{min}}&\ \x{\int_\tau^T 
f^y(t, 
Y_t, 
\law{Y_\tau}, u^y_t(\tau))dt + \mathbb{I}\{\tau = 0\}h^y(Y_0, \law{Y_0})},
  \\
  \text{s.t.}&\ \ dY_t = b^y(t, Y_t, \law{Y_\tau}, Z_t, u^y_t(\tau))dt + Z_ 
tdB_t,
  \\
  &\ \ Y_T = y_T,
  \end{aligned}
 \right.
\end{equation}
where $\mathcal{U}^y_\tau$ is defined in the same way as $\mathcal{U}^y$, but 
with the interval $[0,T]$ replaced by $[\tau, T]$, cf.~\eqref{eq:def_feasible}.
This is an intermediate level between basic and optimal rationality. Pedestrian 
decision making can also be modeled as a decentralized mechanism, i.e. instead 
of cooperating, pedestrians compete in a game-like manner within the crowd. 
Decentralized crowd formation can be modeled by a MFG:
\begin{itemize}
 \item[(i)] Fix a deterministic function $\mu_\cdot: [0,T] \rightarrow 
\mathcal{P}_2(\rd)$.
 \item[(ii)] Solve the stochastic control problem
 \begin{equation}
\label{eq:backward_CP_MFG}
\left\{
 \begin{aligned}
  \underset{u^y_\cdot \in\mathcal{U}^y}{\text{min}}&\ \x{\int_0^T fy(t, 
Y_t, 
\mu_t, u^y_t)dt + h^y(Y_0, \mu_0)},
  \\
  \text{s.t.}&\ \ dY_t = b(t, Y_t, \mu_t, Z_t, u^y_t)dt + Z_ tdB_t,
  \\
  &\ \ Y_T = y_T.
  \end{aligned}
 \right.
\end{equation}
 \item[(iii)] Determine the function $\hat{\mu}_\cdot : [0,T]\rightarrow 
\mathcal{P}_2(\rd)$ such that $\hat{\mu}_t$ is the law of the 
optimally 
controlled state from Step (ii) at time $t$.
\end{itemize}
Furthermore, minimal exit time (evacuation) problems can be posed at all levels 
of rationality. 
\vspace{0.2cm}\\
The numerical simulation of mean-field BFSDE systems is in this paper done 
either with the least-square Monte Carlo method of \cite{bender2013posteriori}, 
or by reducing them to a system of ODEs by the method of matching. 
The downside with the least-square Monte Carlo method is that it is not clear 
which basis functions to use and the matching method is feasible only for 
linear quadratic problems.
Other simulation approaches include deep learning \cite{weinan2017deep} and 
fixed-point schemes \cite{djehich2018MFBFSDE}. Fast and stable numerical 
solvers for mean-field BFSDEs beyond the linear-quadratic case is an area of 
research that would benefit many applied fields. In pedestrian crowd 
modeling, improved solvers would facilitate simulation when effects like 
congestion, crowd aversion, and anisotropic preferences are present.

\appendix

\section{Mean-field BFSDE}
\label{sec:appendixA}
Given a control pair $(u^x_\cdot, u^y_\cdot) \in \mathcal{U}^x\times 
\mathcal{U}^y$, systems of the form \eqref{eq:backward_dynamics} have been 
studied in the context of optimal control of mean-field type, where they 
naturally arise as necessary optimality conditions. This appendix summarizes 
some of the results on existence and uniqueness of solutions to MF-BFSDEs. 
Let
\begin{equation}
\begin{aligned}
 &\mathbb{H}^{2,d} := \left\{ V_\cdot\ \rd\text{-valued and prog. meas.}\ 
:\ \mathbb{E}\left[\int_0^T |V_s|^2ds\right] < \infty\right\}
 \\
 &\mathbb{S}^{2,d} := \left\{ V_\cdot\ \rd\text{-valued and prog.meas.}\ 
:\ \mathbb{E}\left[\sup_{s\in[0,T]} |V_s|^2 \right] < \infty\right\}
\end{aligned}
\end{equation}
and recall that, in the case of a fixed control pair, $b^x$ and $b^y$ are 
functions of
\begin{equation}
 (\omega, t, y, \mu^y, z, x, \mu^x) \in 
\Omega\times[0,T]\times\rd\times\mathcal{P}_2(\rd) \times 
\mathbb{R}^{d\times(w_x+w_y)}\times\rd\times\mathcal{P}(\rd)
\end{equation}
and $\sigma^x$ is a function of
\begin{equation}
 (\omega, y, \mu^y, z, x, \mu^x) \in 
\Omega\times\rd\times\mathcal{P}_2(\rd) \times 
\mathbb{R}^{d\times(w_x+w_y)}\times\rd\times\mathcal{P}(\rd).
\end{equation}

\subsection{Quadratic-type constraints}
For $t\in[0,T]$, $\mu^y, \mu^x\in\mathcal{P}2(\rd)$, $x,\bar{x},y,\bar{y} \in 
\rd$ and $z,\bar{z} \in \mathbb{R}^{d\times (w_x+w_y)}$, let
\begin{equation}
\begin{aligned}
 &\mathcal{A}(t,y, \bar{y}, \mu^y, v,z,\bar{z}, x, \bar{x}, \mu^x) 
 \\
 &=
 (b^y(t,y,\mu^y,v, z, x,\mu^x) - b^y(t,\bar{y},\mu^y, v, \bar{z}, 
\bar{x}, \mu^x))\cdot(y-\bar{y}) 
\\
&\quad + (b^x(t,y,\mu^y,v, z, x,\mu^x) - b^x(t,\bar{y},\mu^y, v, \bar{z}, 
\bar{x}, \mu^x))\cdot(x-\bar{x}) 
\\
&\quad  + (\sigma^x(t,y,\mu^y, z, x,\mu^x) - \sigma^x(t,\bar{y},\mu^y, 
\bar{z}, \bar{x}, \mu^x))\cdot(z-\bar{z}).
\end{aligned}
\end{equation}
In \cite{djehich2018MFBFSDE}, the authors provide conditions on
$\mathcal{A}$ under which, alongside standard assumptions, 
\eqref{eq:backward_dynamics} has a unique solution 
$(X_\cdot, Y_\cdot, Z_\cdot)$ in $\mathbb{H}^{2,d} \times 
\mathbb{H}^{2,d} \times \mathbb{H}^{2,d\times(w_x+w_y)}$ for all $T$.

\subsection{Small time constraint}

Under standard Lipschitz- and linear growth-conditions and for a non-degenerate 
diffusion $\sigma^x$, \eqref{eq:backward_dynamics} has a unique solution 
$(X_\cdot, Y_\cdot, Z_\cdot) \in \mathbb{S}^{2,d}\times 
\mathbb{S}^{2,d}\times \mathbb{H}^{2,d\times (w_x+w_y)}$ for small enough $T$ 
\cite{carmona2018probabilistic}. The bound on $T$ depends on the Lipschitz 
coefficients. In \cite{carmona2018probabilistic}, the 
authors provide an example where uniqueness fails.

\section{Differentiation of measure-valued functions}
\label{sec:appendixMeas}
The differentiation of measure-valued functions is handled with the lifting 
technique, introduced by P.-L. Lions and outlined in for example 
\cite{cardaliaguet2010notes, buckdahn2011general, carmona2018probabilistic}.
Assume that the underlying probability space is 
rich enough, so that for every $\mu \in 
\pp_2(\rd)$ there is a random variable $Y \in 
L^2_{\mathcal{F}}(\Omega; \rd)$ such that $\law{Y} = \mu$. A 
probability space with this property is $([0,1], \mathcal{B}([0,1]), dx)$.  
Under this assumption, any function $f: \pp_2(\rd) \rightarrow \mathbb{R}$ 
induces a function $F : L^2(\rd) \rightarrow 
\mathbb{R}$ so that $F(Y) := 
f(\law{Y})$. The Fr\'echet derivative of $F$ at $Y$, whenever it exists, is the 
continuous linear functional $DF[Y]$ that satisfies
\begin{equation}
\label{eq:def_of_DF}
 F(Y') - F(Y) = \mathbb{E}\left[DF[Y]\cdot(Y'-Y)\right] + o(\|Y'-Y\|_2), \quad 
\forall Y'\in L^2_{\mathcal{F}}(\Omega; \rd).
\end{equation}
Riesz' Representation Theorem yields 
uniqueness of $DF[Y]$. Furthermore, there 
exists a Borel function $\varphi[\mu] : \mathbb{R}^d \rightarrow\mathbb{R}^d$, 
independent of the version of $Y$, such that $DF[Y] = \varphi[\law{Y}](Y)$
\cite{cardaliaguet2010notes}. Therefore
\begin{equation}
\label{eq:def_of_taylor}
 f(\law{Y'}) - f(\law{Y}) = \mathbb{E}\left[\varphi[Y](Y)(Y'-Y)\right] + 
o(\|Y'-Y\|), \quad \forall Y'\in L^2_{\mathcal{F}}(\Omega; \rd).
\end{equation}
Denote $\partial_\mu f(\mu;x) := \varphi[\mu](x)$, $x\in\rd$, and 
$\partial_\mu f(\law(Y); Y) =: \partial_\mu f(\law(Y))$. The following identity 
characterizes derivatives with respect to elements in $\mathcal{P}_2(\rd)$,
\begin{equation}
 DF[Y] = \varphi[\law{Y}](Y) = \partial_\mu f(\law{Y}).
\end{equation}
Equation \eqref{eq:def_of_taylor} is the Taylor approximation of a 
measure-valued function. Consider now an $f$ that besides the 
measure takes another argument, $\xi$. Then 
\begin{equation}
 f(\xi, \law{Y'}) - f(\xi,\law{Y}) = \mathbb{E}\left[ \partial_\mu 
f(\widetilde{\xi}, \law{Y}; Y)(Y'-Y)\right] + o(\|Y-Y'\|_2),
\end{equation}
where the expectation is taken over \textit{non-tilded random 
variables}. This is abbreviated as
\begin{equation}
\label{eq:def_of_poststar}
 \mathbb{E}\left[\partial_\mu f(\widetilde{\xi}, \law{Y}; Y)(Y'-Y)\right] =: 
\mathbb{E}\left[(\partial_\mu f(\xi, \law{Y}))^*(Y'-Y)\right].
\end{equation}
Note that $\law{Y}$ is deterministic, so the expectation is only taken over the 
'directional argument' of $\partial_\mu f$, $Y$. Also, the expected value 
\eqref{eq:def_of_poststar} is stochastic, since it is \textit{not taken over} 
$\xi$. Taking another expectation and changing the order of integration yields
\begin{equation}
 \mathbb{E}\left[\widetilde{\mathbb{E}}\left[\partial_\mu f(\widetilde{\xi}, 
\law{Y}; Y)\right] (Y'-Y)\right],
\end{equation}
where the tilded expectation is taken \textit{over tilded random variables}. 
This is abbreviated as
\begin{equation}
 \label{eq:def_of_prestar}
 \widetilde{E}\left[\partial_\mu f(\widetilde{\xi}, \law{Y}; Y)\right] =: 
\mathbb{E}\left[ ^*( \partial_\mu f(\xi, \law{Y})) \right]
\end{equation}
\noindent
\textbf{Example from Section \ref{sec:keeptogether}}
The following measure derivative appears in Section 3.1,
\begin{equation}
 \mathbb{E}\left[ ^*\left( \partial_{\mu}
\left(Y_t - \int_{\mathbb{R}^2}y\law{Y_t}(dy)\right)^2\right)\right].
\end{equation}
Note that $\left(Y_t - \int_{\mathbb{R}^2}y\law{Y_t}(dy)\right)^2 = 
(Y_t -\x{M})^2 =: F(M)$ where $M$ is a random variable with 
probability law $\law{Y_t}$. By the Taylor expansion, $DF[M] = 
2(Y_t - \x{M})$ and therefore 
\begin{equation}
 \mathbb{E}\left[ ^*\left( \partial_{\mu}
\left(Y_t - \int_{\mathbb{R}^2}y\law{Y_t}(dy)\right)^2\right)\right] = \x{2(Y_t 
- \x{M})} = 0.
\end{equation}

\section{Proof of Theorem~\ref{thm:game_smp}}
\label{sec:appendixProof}
Let $\bar{u}^{x,\epsilon}_\cdot$ be a spike 
variation of $\hat{u}^x_\cdot$,
\begin{equation}
\bar{u}^{x,\epsilon}_t :=
\left\{
\begin{array}{l l}
\hat{u}^x_t, & t \in [0,T]\backslash E_\epsilon,
\\
u_t, & t \in E_\epsilon,
\end{array}
\right.
\end{equation}
where $u_\cdot \in \mathcal{U}^x$ and $E_\epsilon$ is a subset of 
$[0,T]$ of measure $\epsilon$. Given the control pair 
$(\bar{u}^{x,\epsilon}_\cdot, \hat{u}^y_\cdot)$, denote the corresponding 
solution to the state equation \eqref{eq:backward_dynamics} by 
$\bar{X}^\epsilon_\cdot$ and $(\bar{Y}^\epsilon_\cdot, \bar{Z}^\epsilon_\cdot)$. 
To ease notation, let for $\vartheta \in \{b^x, b^y, \sigma^x, f^x, f^y, h^x, 
h^y\}$,
\begin{equation}
\label{eq:def_for_notation}
  \bar{\vartheta}^\epsilon_t := \vartheta(t, \bar{\theta}^{y,\epsilon}_t, 
\hat{u}^y_t, \bar{Z}^\epsilon_t, \bar{\Theta}^{x,\epsilon}_t),
\ \
\hat{\vartheta}^\epsilon_t := \vartheta(t,\hat{\Theta}^{y}_t,\hat{Z}_t, 
\hat{\Theta}^x_t), 
\ \
\delta_x \vartheta(t) := \vartheta(t, \hat{\Theta}^y_t, \hat{Z}_t, 
\hat{\theta}^x_t, \bar{u}^{x,\epsilon}_t) - \hat{\vartheta}^{\epsilon}_t.
\end{equation}
Consider the ordinary pedestrian's potential loss, would she switch from the 
equilibrium control $\hat{u}^x_\cdot$ to the perturbed 
$\bar{u}^{x,\epsilon}_\cdot$,
\begin{equation}
\begin{aligned}
J^x(\bar{u}^{x,\epsilon}_\cdot; \hat{u}^y_\cdot) - J(\hat{u}^x_\cdot; 
\hat{u}^y_\cdot) 
&= 
\mathbb{E}\left[\int_0^T \left(\bar{f}^{x,\epsilon}_t - 
\hat{f}^x_t\right)dt +  \bar{h}^{x,\epsilon}_T - 
\hat{h}^x_T \right].
\end{aligned}
\end{equation}
A Taylor expansion of the terminal cost difference yields
\begin{equation}
\label{eq:variation_of_h} 
\mathbb{E}\left[\bar{h}^{x,\epsilon}_T  - \hat{h}^x_T\right]
=
\mathbb{E}\left[\left(\partial_y \hat{h}^x_t + 
\mathbb{E}[^*(\partial_\mu\hat{h}^x_T)]\right)\left(\bar{X}^{\epsilon
}_t - \hat{X}_t\right)\right]
+ o\left(\|\bar{X}^\epsilon_T - \hat{X}_T\|_2\right),
\end{equation}
Let $\widetilde{X}^x_\cdot$ and $(\widetilde{Y}^x_\cdot, 
\widetilde{Z}^x_\cdot)$ be the first order variation processes, solving the 
linear BFSDE system
\begin{equation}
\label{eq:variation_eq_xx}
\left\{
\begin{aligned}
d\widetilde{X}^x_t &= \Big\{ \left(\partial_y\hat{b}^x_t + 
\mathbb{E}[^*(\partial_{\mu^y}\hat{b}^x_t)]\right)\widetilde{Y}^x_t + 
\partial_z\hat{b}^x_t\widetilde{Z}^x_t + \left(\partial_x\hat{b}^x_t +
\mathbb{E}[^*(\partial_{\mu^x}\hat{b}^x_t)]\right)\widetilde{X}^x_t
\\\
&\hspace{1cm} + \delta_x b^x(t)1_{E_\epsilon}(t)\Big\}dt
\\
&\quad + \Big\{ \left(\partial_y\hat{\sigma}^x_t + 
\mathbb{E}[^*(\partial_{\mu^y}\hat{\sigma}^x_t)]\right)\widetilde{Y}^x_t + 
\partial_z\hat{\sigma}^x_t\widetilde{Z}^x_t + \left(\partial_x\hat{\sigma}^x_t +
\mathbb{E}[^*(\partial_{\mu^x}\hat{\sigma}^x_t)]\right)\widetilde{X}^x_t \Big\}
dB^x_t,
\\
d\widetilde{Y}^x_t &= \Big\{ \left(\partial_y\hat{b}^y_t + 
\mathbb{E}[^*(\partial_{\mu^y}\hat{b}^y_t)]\right)\widetilde{Y}^x_t + 
\partial_z\hat{b}^y_t\widetilde{Z}^x_t + \left(\partial_x\hat{b}^y_t +
\mathbb{E}[^*(\partial_{\mu^x}\hat{b}^y_t)]\right)\widetilde{X}^x_t
\\\
&\hspace{1cm} + \delta_x b^y(t)1_{E_\epsilon}(t)\Big\}dt + 
\widetilde{Z}^x_tdB_t,
\\
\widetilde{X}^x_0 &= 0, \quad \widetilde{Y}^x_T = 0.
\end{aligned}
\right.
\end{equation}
\begin{lemma}
\label{lem:estimate1}
Assume that $b^x$ are $b^y$ are Lipschitz in the controls, that $b^x, b^y, f^x, 
f^y$ and $\sigma^x$ are differentiable at the equilibrium point almost surely 
for all 
$t$, that their derivatives are bounded almost surely for all $t$ and that
\begin{equation}
 \partial_x\hat{h}^x + \mathbb{E}\left[^*(\partial_{\mu^x}\hat{h}^x)\right] 
\in L^2_{\mathcal{F}_T}(\rd),\quad
 \partial_y\hat{h}^y + \mathbb{E}\left[^*(\partial_{\mu^y}\hat{h}^y)\right] 
\in L^2_{\mathcal{F}_0}(\rd).
\end{equation} 
Then for some positive constant $C$,
\begin{equation}
\begin{aligned}
  &\sup_{t\in[0,T]} \mathbb{E}\left[ |\widetilde{X}^x_t|^2 + 
|\widetilde{Y}^x_t|^2 + \int_0^t\|\widetilde{Z}^x_s\|_F\; ds \right] \leq 
C\epsilon^2. 
  \\
  &\sup_{t\in[0,T]} \mathbb{E}\left[ 
|\bar{X}^{\epsilon}_t - \hat{X}_t - \widetilde{X}^x_t|^2 + 
|\bar{Y}^\epsilon_t - \hat{Y}_t - \widetilde{Y}^x_t|^2 + 
\int_0^t\|\bar{Z}^{\epsilon}_s - \hat{Z}_s - \widetilde{Z}^x_s\|_F\; ds \right] 
\leq 
C\epsilon^2. 
\end{aligned}
\end{equation}
\end{lemma}
\begin{proof}
 The proof is a combination of standard estimates, see 
\cite{yong1999stochastic} for the SDE terms and \cite{aurell2018mean} for the 
BSDE terms.
\end{proof}
\noindent
The adjoint processes and the 
Hamiltonian, defined in \eqref{eq:adjoint_in_thm} and 
\eqref{eq:def_hamiltonian} respectively, yield together with 
Lemma~\ref{lem:estimate1} and integration by parts that
\begin{equation}
\label{eq:C7}
\begin{aligned}
&J^x(\bar{u}^{x,\epsilon}_\cdot; \hat{u}^y_\cdot) - J(\hat{u}^x_\cdot; 
\hat{u}^y_\cdot) 
= 
\mathbb{E}\Bigg[\int_0^T \Bigg\{\left(\bar{b}^{x,\epsilon}_t - 
\hat{b}^x_t\right)p^{xx}_t + \left(\bar{b}^{y,\epsilon}_t - 
\hat{b}^y_t\right)p^{xy}_t 
\\
&\hspace{2cm} + \left(\bar{\sigma}^{x,\epsilon}_t - 
\hat{\sigma}^x_t\right)q^{xx}_t - \left(\bar{H}^{x,\epsilon}_t - 
\hat{H}^x_t\right)\Bigg\} dt - p^{xx}_T\widetilde{X}^x_T \Bigg] + o(\epsilon)
\\
&=
\mathbb{E}\Bigg[\int_0^T \Bigg\{\left(\bar{b}^{x,\epsilon}_t - 
\hat{b}^x_t\right)p^{xx}_t + \left(\bar{b}^{y,\epsilon}_t - 
\hat{b}^y_t\right)p^{xy}_t + \left(\bar{\sigma}^{x,\epsilon}_t - 
\hat{\sigma}^x_t\right)q^{xx}_t - \left(\bar{H}^{x,\epsilon}_t - 
\hat{H}^x_t\right)\Bigg\} dt
\\
&\hspace{2cm}  - \int_0^T \widetilde{X}^x_t dp^{xx}_t - \int_0^T p^{xx}_t 
d\widetilde{X}^x_t - \int_0^T d\langle p^{xx}_\cdot, 
\widetilde{X}^x_\cdot\rangle_t
\\
&\hspace{2cm}  - \int_0^T \widetilde{Y}^x_t dp^{xy}_t - \int_0^T p^{xy}_t 
d\widetilde{Y}^x_t - \int_0^T d\langle p^{xy}_\cdot, 
\widetilde{Y}^x_\cdot\rangle_t \Bigg] + o(\epsilon)
\\
&=
-\mathbb{E}\left[ \int_0^T \delta_x H^x(t)1_{E_\epsilon}(t)dt \right] + 
o(\epsilon),
\end{aligned}
\end{equation}
where $\bar{H}^{x,\epsilon}_t, \hat{H}^x_t$ and $\delta_xH^x(t)$ are defined in 
line with~\eqref{eq:def_for_notation}. The final equality is retrieved by 
expanding all differences on the third row of \eqref{eq:C7}, canceling all but 
$\delta_xH^x(t)1_{E_\epsilon}(t)$ with the forth and fifth row, while making use 
of the estimates from Lemma~\ref{lem:estimate1}. 
\vspace{0.2cm}\\
Consider now a spike variation of the tagged's control,
\begin{equation}
\check{u}^{y,\epsilon}_t :=
\left\{
\begin{array}{l l}
\hat{u}^y_t, & t \in [0,T]\backslash E_\epsilon,
\\
u_t, & t \in E_\epsilon,
\end{array}
\right.
\end{equation}
where $u_\cdot \in \mathcal{U}^y$. Following the same lines of calculations as 
above, one finds that if $p^{yx}_\cdot$, $p^{yy}_\cdot$ are given by the 
adjoint equations \eqref{eq:adjoint_in_thm} and the Hamiltonian $H^y$ 
by~\eqref{eq:def_hamiltonian}, then
\begin{equation}
 J^y(\check{u}^{y,\epsilon}_\cdot; \hat{u}^x_\cdot) - J^y(\hat{u}^y_\cdot; 
\hat{u}^x_\cdot) = -\mathbb{E}\left[\int_0^T \delta_y H^y(t)1_{E_\epsilon} 
dt\right] + o(\epsilon),
\end{equation}
where $\delta_y H^y(t)$ is defined in line with~\eqref{eq:def_for_notation}, 
for the spike perturbation $\check{u}^{y,\epsilon}_t$.
The rest of the proof is standard, and can be found in for example 
\cite{yong1999stochastic}.

 \bibliographystyle{plain}
 \bibliography{references}

\end{document}